\documentclass[a4paper]{article}

\usepackage{amssymb}
\usepackage{amsmath}
\usepackage{amsthm}
\usepackage{amsfonts}
\usepackage{graphicx}
\usepackage{subfigure}
\usepackage{enumerate}
\usepackage{color}
\usepackage{srcltx}

\theoremstyle{plain}
\newtheorem{nn}{}[section]
\newtheorem{theorem}[nn]{Theorem}
\newtheorem{lemma}[nn]{Lemma}

\newtheorem{proposition}[nn]{Proposition}

\newtheorem{remark}[nn]{Remark}

\newcommand{\F}{\mathcal{F}}
\newcommand{\R}{\mathbb{R}}

\newcommand{\Z}{\mathbb{Z}}

\newcommand{\intt}{\mathrm{int}}
\newcommand{\relintt}{\mathrm{relint}}

\newcommand{\M}{\mathcal{M}^3}
\newcommand{\vertt}{\mathrm{vert}}
\newcommand{\floor}[1]{\left\lfloor #1 \right\rfloor}

\newcommand{\area}{\mathrm{A}}

\newcommand{\rmcmd}[1]{\mathop{\mathrm{#1}}}
\newcommand{\integer}{\mathbb{Z}}
\newcommand{\cP}{\mathcal{P}}
\newcommand{\cPi}{\cP_{\rmcmd{i}}}

\newcommand{\cPif}[2]{\cP_{\rmcmd{if}}^{#2}(#1)}
\newcommand{\cPifm}[2]{\cP_{\rmcmd{ifm}}^{#2}(#1)}
\newcommand{\cPfmi}[2]{\cP_{\rmcmd{fmi}}^{#2}(#1)}
\newcommand{\cC}{\mathcal{C}}

\newcommand{\cCfm}[2]{\cC_{\rmcmd{fm}}^{#2}(#1)}
\newcommand{\aff}{\rmcmd{aff}}
\newcommand{\vx}{\rmcmd{vert}}
\newcommand{\rec}{\rmcmd{rec}}
\newcommand{\relbd}{\rmcmd{relbd}}
\newcommand{\relintr}{\rmcmd{relint}}
\newcommand{\conv}{\rmcmd{conv}}

\newcommand{\real}{\mathbb{R}}
\newcommand{\term}[1]{\emph{#1}}
\newcommand{\natur}{\mathbb{N}}
\newcommand{\Aff}{\rmcmd{Aff}}
\newcommand{\intr}{\rmcmd{int}}
\newcommand{\modulo}[1]{\, (\mathrm{mod} \, #1)}

\newcommand{\vol}{\rmcmd{vol}}

\newcommand{\lin}{\rmcmd{lin}}
\newcommand{\cX}{\mathcal{X}}
\newcommand{\ray}[1]{\left[#1\right>}
\newcommand{\thmcaption}[1]{\emph{(#1.)}}
\newcommand{\ceiling}[1]{\left\lceil#1\right\rceil}
\newcommand{\cR}{\mathcal{R}}

\newcommand{\card}[1]{\left|#1\right|}

\newcommand{\comment}[1]{}

\title{Maximal lattice-free polyhedra: \\
  finiteness and an explicit description in dimension three}

\author{Gennadiy Averkov \and Christian Wagner \and Robert
  Weismantel}

\begin{document}

\maketitle

\begin{abstract}
A convex set with nonempty interior is maximal lattice-free
if it is inclusion-maximal with respect to the property of
not containing integer points in its interior.
Maximal lattice-free convex sets are known to be polyhedra.
The precision of a rational polyhedron $P$ in $\R^d$ is
the smallest natural number $s$ such that $sP$ is an
integral polyhedron.
In this paper we show that, up to affine mappings
preserving $\Z^d$, the number of maximal lattice-free
rational polyhedra of a given precision $s$ is finite.
Furthermore, we present the complete list of all maximal
lattice-free integral polyhedra in dimension three.
Our results are motivated by recent research on cutting
plane theory in mixed-integer linear optimization.
\end{abstract}

\newtheoremstyle{itsemicolon}{}{}{\mdseries\rmfamily}{}{\itshape}{:}{ }{}
\newtheoremstyle{itdot}{}{}{\mdseries\rmfamily}{}{\itshape}{:}{ }{}
\theoremstyle{itdot}
\newtheorem*{msc*}{2010 MSC} 

\begin{msc*}
  52B20, 52C07, 90C10, 90C11
\end{msc*}
% 52C07 (lattice and convex bodies in $n$ dimensions)
% 52B20 (lattice polytopes)
% 90C10 (integer programming)
% 90C11 (mixed integer programming)

\newtheorem*{keywords*}{Keywords}

\begin{keywords*}
cutting plane, integral polyhedron, maximal lattice-free
set, mixed-integer programming 
\end{keywords*}

%%%%%%%%%%%%%%%%%%%%%%%%%%%%%%%%%%%%%%%%%%%%%%%%%%%%%%%%%%%%%%%%%%%%%%%
%% Introduction %%
%%%%%%%%%%%%%%%%%%%%%%%%%%%%%%%%%%%%%%%%%%%%%%%%%%%%%%%%%%%%%%%%%%%%%%%

\section{Introduction}

A convex set $K \subseteq \R^d$ with nonempty interior is
called \emph{lattice-free} if the interior of $K$ does not
contain a point of $\Z^d$ and \emph{maximal lattice-free} if
$K$ is inclusion-maximal in the class of lattice-free convex
sets (for different definitions of lattice-freeness see, for
instance, \cite{Reznick86,Scarf85,Seboe99}).
Every maximal lattice-free set $K$ is a polyhedron with an
integer point in the relative interior of each facet of $K$
(see, for instance, \cite[Proposition~3.3]{Lovasz89}).

The study of maximal lattice-free polyhedra is motivated by
recent research in mixed-integer linear optimization.
Cutting planes for mixed-integer linear programs can be
obtained from a simultaneous consideration of several rows
of a simplex tableau (see, for instance,
\cite{anlowe2,AndersenLouveauxWeismantelWolsey07,AndersenWagnerWeismantel09,Balas71,BasuBonamiCornuejolsMargot09,boco,CornuejolsMargot08,DeyRichard08,DeyWolsey08,Espinoza08,Zambelli09}).
Such cutting planes are deducible from lattice-free convex
sets.
Furthermore, the strongest cutting planes are derived from
maximal lattice-free polyhedra. 
It is therefore natural to ask for a characterization of
maximal lattice-free polyhedra.
Since we aim at algorithmic applications, we restrict
considerations to the class of maximal lattice-free rational
polyhedra.
In this paper we answer the following two questions:
\begin{enumerate}[I.]
  \item Given the dimension $d \in \natur$ and the precision
    of a maximal lattice-free rational polyhedron $P$, how
    many different shapes are possible for $P$?
  \item How do maximal lattice-free integral polyhedra in
    dimension three look like? 
\end{enumerate}  
The answer to the first question is that $P$ can only have
finitely many shapes (a precise formulation will be given in
the following section).
In particular, we prove the following result.

\begin{theorem} \label{explicit.statement}
Let ${\cal I}^d$ denote the set of all lattice-free integral
polyhedra $P \subseteq \real^d$  such that $P$ is not
properly contained in another lattice-free integral
polyhedron.
Then there exists a constant $N$ depending only on $d$, and
polyhedra $P_1,\ldots,P_N \in {\cal I}^d$ such that for
every $P \in {\cal I}^d$ one has  $P = UP_j + v$ for some
unimodular matrix $U \in \Z^{d \times d}$, integral vector
$v \in \integer^d$, and $j \in \{1, \dots, N\}$.
\end{theorem}

The proof of Theorem \ref{explicit.statement} suggests that
the constant $N$ grows rapidly in $d$.
Moreover, the proof does not imply any quick constructive
procedure for enumeration of the polyhedra
$P_1,\ldots,P_N$.
Having applications in mixed-integer cutting plane theory in
mind, it is thus desirable to provide  a precise
classification for small dimensions.
Notice that finite termination of a cutting plane algorithm
only requires cutting planes associated with lattice-free
integral polyhedra as derived in
Theorem~\ref{explicit.statement} (see
\cite{BasuCornuejolsMargot10,DeyLouveaux09,DelPiaWeismantel10}).
The explicit description of ${\cal I}^d$ in Theorem
\ref{explicit.statement} for $d = 1,2$ is folklore.
However, already the class ${\cal I}^3$ is rather complex.
Thus, the complete enumeration of ${\cal I}^d$ for an
arbitrary $d \ge 3$ is challenging.
We provide a classification of an important subclass of
${\cal I}^3$.

Theorem \ref{explicit.statement} follows directly from
Theorem \ref{finiteness} by setting $s = 1$.
The classification of a subclass of ${\cal I}^3$ is stated
in Theorem \ref{main.thm}.

%%%%%%%%%%%%%%%%%%%%%%%%%%%%%%%%%%%%%%%%%%%%%%%%%%%%%%%%%%%%%%%%%%%%%%%
%% Main Results and Notation %%
%%%%%%%%%%%%%%%%%%%%%%%%%%%%%%%%%%%%%%%%%%%%%%%%%%%%%%%%%%%%%%%%%%%%%%%

\section{Main results and notation}

Let us first introduce the notation used in the formulations
of our main results.
(Introduction of the standard notation is postponed to the
end of this section.)
For the relevant background information in convex geometry,
in particular with respect to polyhedra and lattices, we
refer to the books
\cite{barvinok-2002,gruber-2007,gruber-lekkerkerker-1987,Rockafellar72}.

The intersection of finitely many closed halfspaces is said
to be a \term{polyhedron}.
By $\cP^d$ we denote the set of all polyhedra in $\real^d$
(where the elements of $\cP^d$ do not have to be
full-dimensional).
A bounded polyhedron is called a \term{polytope}.
A polyhedron $P \in \cP^d$ is said to be \term{integral} if
$P = \conv(P \cap \integer^d)$;
and $P$ is said to be \emph{rational} if $sP := \{sx \in
\real^d : x \in P\}$ is an integral polyhedron for some
finite integer $s \ge 1$.
The \emph{precision} of a rational polyhedron $P$ is the
smallest integer $s \ge 1$ such that $sP$ is an integral
polyhedron.

If $\Lambda$ is a lattice in $\real^d$, then a polyhedron $P
\in \cP^d$ is said to be \term{$\Lambda$-free} if $\intr(P)
\cap \Lambda = \emptyset$.
For $\Lambda=\integer^d$ we say ``lattice-free'' rather than
``$\Lambda$-free''.
In this paper, we restrict $\Lambda$ to be $s \integer^d$
for some $s \in \natur$. 

Our results are concerned with the interplay of the
following three properties of polyhedra:
integrality (abbreviated with ``i''),
$\Lambda$-freeness (abbreviated with ``f'' and an additional
``s'' in brackets to indicate the dependency on $\Lambda = s
\integer^d$), and inclusion-maximality in a given class
(abbreviated with ``m''). 
By $\cPi^d$ we denote the set of integral polyhedra
belonging to $\cP^d$, by $\cPif{s}{d}$ the set of
$\Lambda$-free polyhedra belonging to $\cPi^d$, and by
$\cPifm{s}{d}$ the set of elements of $\cPif{s}{d}$ which
are maximal within $\cPif{s}{d}$ with respect to inclusion.

Let $\Aff(\Lambda)$ denote the group of all affine 
transformations $T$ in $\real^d$ with $T(\Lambda) =
\Lambda$.
It is not hard to see that $\Aff(\Lambda) \subseteq
\Aff(\integer^d)$.
Henceforth, the transformations in $\Aff(\Lambda)$ are
called \term{$\Lambda$-preserving}, while the
transformations in $\Aff(\integer^d)$ are called
\emph{unimodular}.
If a set $P$ can be mapped to a set $Q$ by a
$\Lambda$-preserving transformation we simply say that both
sets are \emph{equivalent}.
The group $\Aff(\Lambda)$ has a natural action on $\cPi^d$.
Typically, we are interested in polyhedra in $\cPi^d$
identified modulo $\Aff(\Lambda)$, since this identification
does not change affine properties of integral polyhedra
relative to the lattice $\Lambda$.
In particular, two polyhedra $P,Q \in \cPi^d$ which coincide
up to an affine transformation in $\Aff(\Lambda)$ contain
the same number of lattice points in $\integer^d$ and
$\Lambda$ on corresponding faces.

Let us assume that $P \in \cP^d$ is a maximal lattice-free
rational polyhedron with precision $s$.
Thus, $sP$ is an integral polyhedron and the maximality and
lattice-freeness of $P$ with respect to the standard lattice
$\integer^d$ transfers one-to-one into a maximality and
$\Lambda$-freeness of $sP$ with respect to the lattice
$\Lambda = s \integer^d$.
Thus, instead of analyzing ``maximal lattice-free rational
polyhedra''~(which correspond to cutting planes when
rational data is assumed) we can equivalently consider the
more convenient set of ``maximal $\Lambda$-free integral
polyhedra''.
Indeed, from an analytical point of view, the latter set is
easier to handle since results from the literature can be
used which are stated in terms of integral polyhedra.
We are now ready to present our first main result.

\begin{theorem} \label{finiteness}
Let $d,s \in \natur$.
Then $\cPifm{s}{d} / \Aff(\Lambda)$ is a finite set.
\end{theorem}

We now relate maximal $\Lambda$-free integral polyhedra to
the set $\cPifm{s}{d}$.
Let $\cCfm{s}{d}$ be the class of all $\Lambda$-free convex
sets in $\real^d$ which are not properly contained in
another $\Lambda$-free convex set.
The elements of $\cCfm{s}{d}$ are polyhedra (see
\cite[Proposition\,3.3]{Lovasz89}).
Thus, $\cCfm{s}{d}$ is the class of all maximal
$\Lambda$-free polyhedra in $\R^d$.
Let $\cPfmi{s}{d} := \cPi^d \cap \cCfm{s}{d}$ be the class of
all maximal $\Lambda$-free integral polyhedra in $\R^d$.
By definition we have $\cPfmi{s}{d} \subseteq \cPifm{s}{d}$.
Both classes, $\cPfmi{s}{d}$ and $\cPifm{s}{d}$, are of
interest in cutting plane theory (see
\cite{DelPiaWeismantel10}).
In particular, the complete characterization of pairs of $s$
and $d$ for which the equality $\cPfmi{s}{d} = \cPifm{s}{d}$
holds is unknown.
For $d=1, s \ge 1$ and $d=2, s=1$ equality can be verified
in a straightforward way.
On the other hand, for $d \ge 2, s \ge 3$ the inclusion is
strict.
For instance, consider the polyhedron
$Q^d_s := \conv(\{o, (2s+1) e_1, (2s+1) e_1 + e_2,
(2s-1) e_1 + (2s-1) e_2\}) + \lin(\{e_3, \dots, e_d\})$.
It is easy to verify that $Q^d_s \in \cPifm{s}{d} \setminus
\cPfmi{s}{d}$.
The remaining cases (that is, $d=2, s=2$ and $d \ge 3, 1 \le
s \le 2$) are open.

The finiteness of $\cPfmi{s}{d} / \Aff(\Lambda)$ follows
directly from Theorem~\ref{finiteness}.
This has two consequences:
First, if we choose $s = 1$, then for every dimension $d$,
up to unimodular transformations, there is only a finite
number of maximal lattice-free integral polyhedra.
Second, if we fix some integer $s \ge 1$ and
consider the set of polytopes with vertices in
$\frac{1}{s} \integer^d$, then there is only a finite number
of maximal lattice-free polytopes in this set up to an
affine transformation preserving $\integer^d$. 

The second part of the paper deals with the classification
of the set $\cPfmi{1}{3}$.
As we show later (in Proposition \ref{S=P+L}), we can
restrict ourselves to polytopes within $\cPfmi{1}{3}$.
Let ${\cal M}^d$ be the set of all maximal lattice-free
integral polytopes in $\R^d$.
In dimension one, the set ${\cal M}^1$ consists of all
intervals $[n,n+1]$ for an integer $n$.
Thus, up to a unimodular transformation, $[0,1]$ is the only
maximal lattice-free integral polytope.
In dimension two, it is easy to see that every element of
${\cal M}^2$ is equivalent to $\conv(\{o,2e_1,2e_2\})$.
In \cite{anwawe} it has been shown that, up to a unimodular
transformation, there are only seven different simplices in
$\M$.
In this paper we complete the classification of elements of
$\M$ by proving the following theorem.

\begin{theorem} \label{main.thm}
Let $P \in \M$.
Then, up to a unimodular transformation, $P$ is one of the
following polytopes (see Figure~\ref{fig.main.thm}):
\begin{itemize}
  \item one of the seven simplices
    \begin{center}
      $\begin{array}{l}
        M_1 = \conv(\{o,2e_1,3e_2,6e_3\}), \\
        M_2 = \conv(\{o,2e_1,4e_2,4e_3\}), \\
        M_3 = \conv(\{o,3e_1,3e_2,3e_3\}), \\
        M_4 = \conv(\{o,e_1,2e_1+4e_2,3e_1+4e_3\}), \\
        M_5 = \conv(\{o,e_1,2e_1+5e_2,3e_1+5e_3\}), \\
        M_6 = \conv(\{o,3e_1,e_1+3e_2,2e_1+3e_3\}), \\
        M_7 = \conv(\{o,4e_1,e_1+2e_2,2e_1+4e_3\}),
      \end{array}$
    \end{center}
  \item the pyramid $M_8 = \conv(B \cup \{a\})$ with the base
        $B = \conv(\{ \pm 2e_1, \pm 2e_2 \})$ and the apex 
        $a = (1,1,2)$,
  \item the pyramid $M_9 = \conv(B \cup \{a\})$ with the base
        $B = \conv(\{-e_1,-e_2,2e_1,2e_2\})$ and the
        apex $a = (1,1,3)$,
  \item the prism $M_{10} = \conv (B \cup (B+u))$ with the
        bases $B$ and $B+u$, where $B = \conv(\{ e_1, e_2,
        -(e_1+e_2) \})$ and $u =  (1,2,3)$,
  \item the prism $M_{11}= \conv(B \cup (B+u))$ with the
        bases $B$ and $B+u$, where $B = \conv(\{ \pm e_1,
        2e_2\})$ and $u = (1,0,2)$,
  \item the parallelepiped $M_{12} = \conv(\{\sigma_1 u_1 +
         \sigma_2 u_2 + \sigma_3 u_3: \sigma_1, \sigma_2,
         \sigma_3 \in \{0,1\}\})$ where $u_1 = (-1,1,0)$,
         $u_2 = (1,1,0)$, and $u_3 = (1,1,2)$.
\end{itemize}

\begin{figure}[ht]
  \centering

  \subfigure[$M_1$ \label{M1}]{\includegraphics[height=2.5cm]{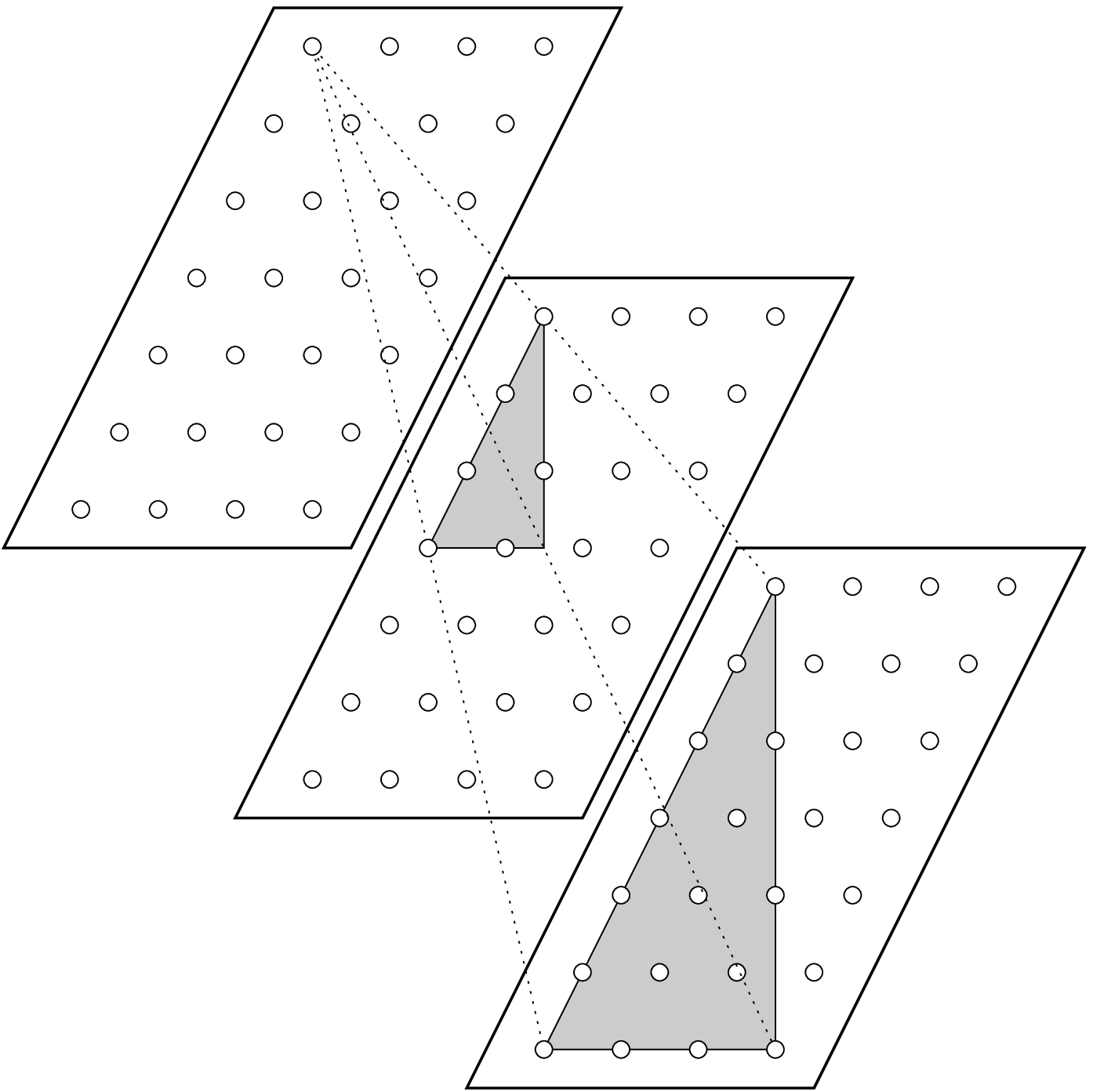}}
  %\\
  \qquad
  \subfigure[$M_2$ \label{M2}]{\includegraphics[height=2.5cm]{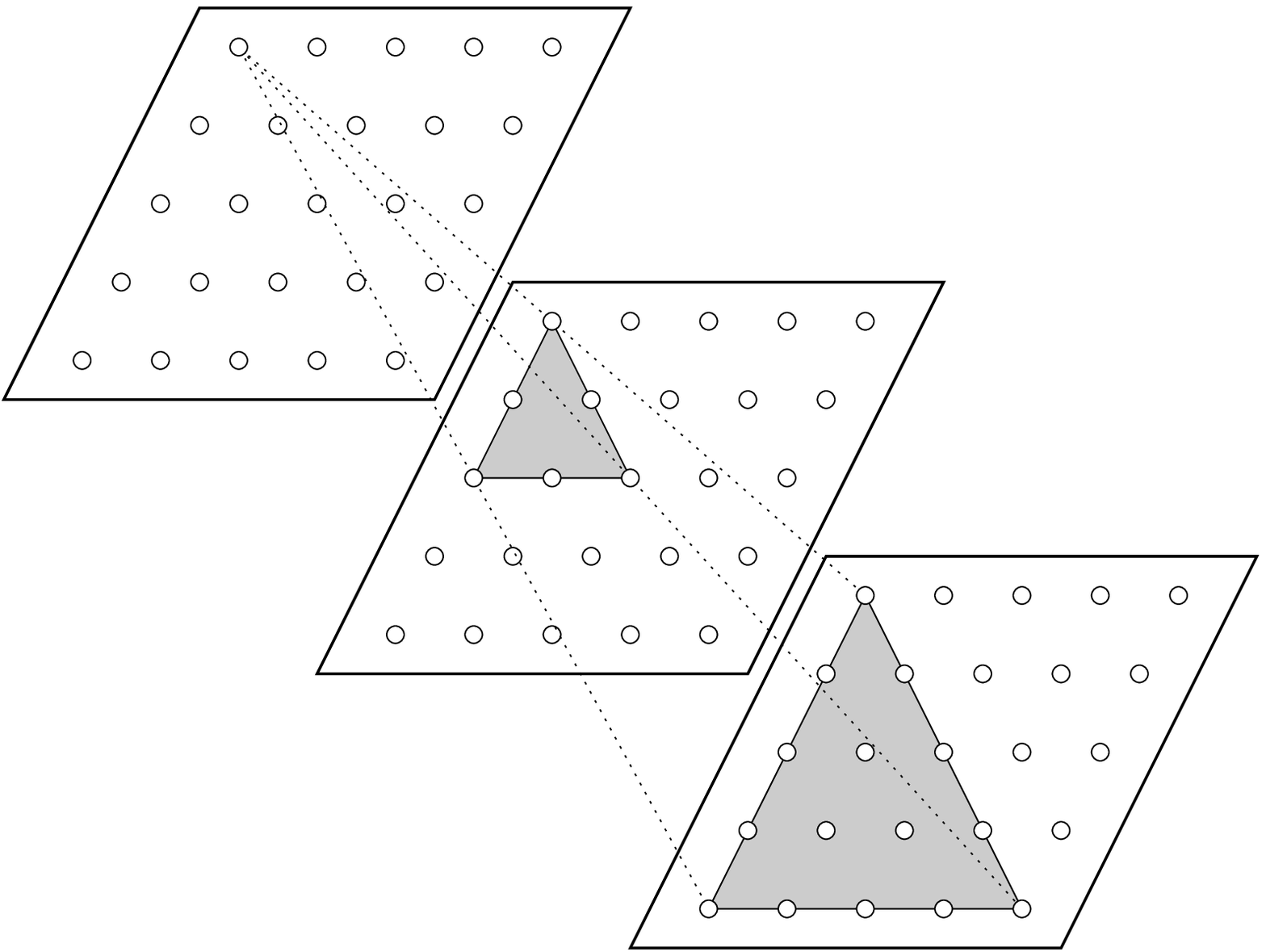}}
  %\\
  \qquad
  \subfigure[$M_3$ \label{M3}]{\includegraphics[height=2.5cm]{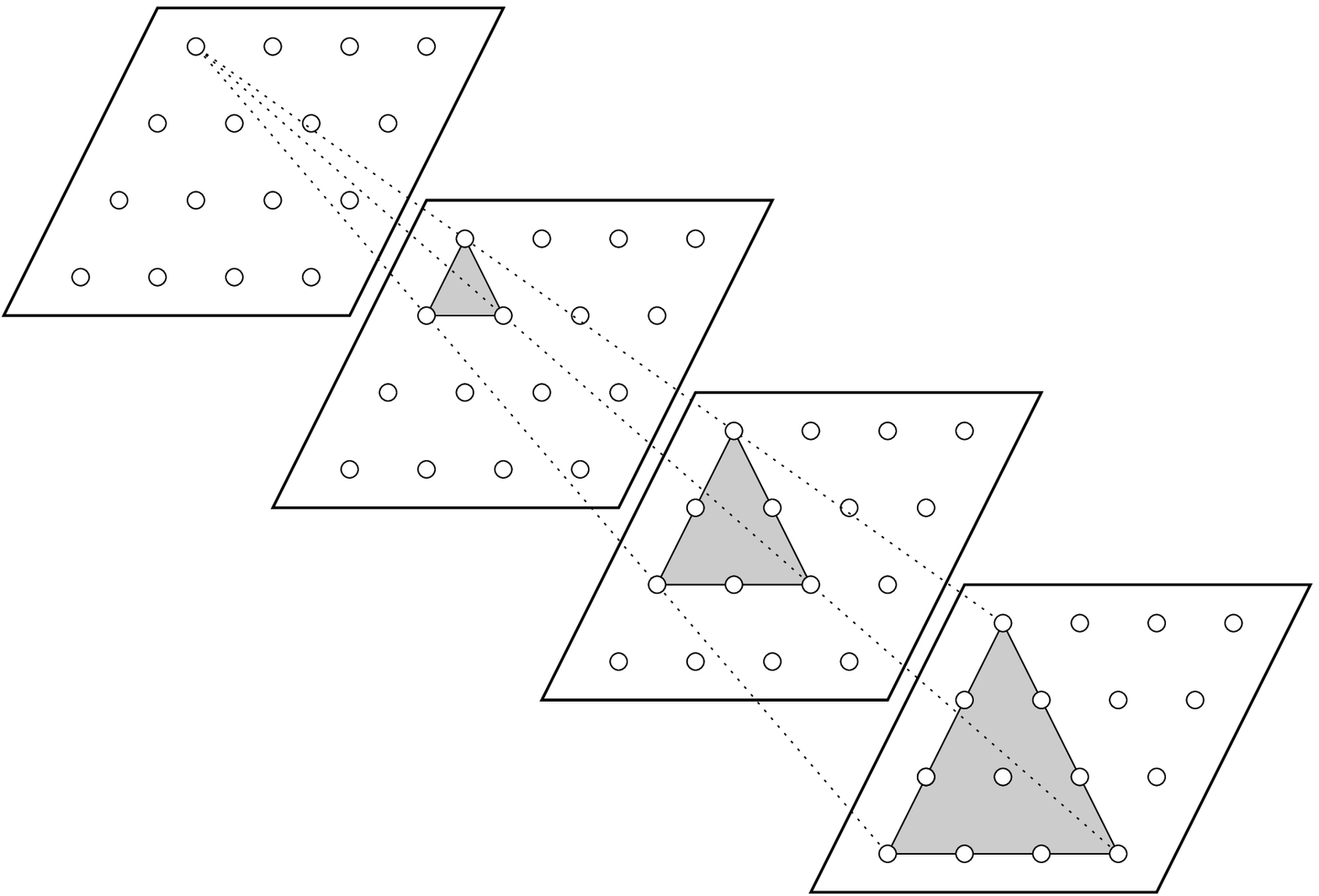}}
  \\
  %\qquad
  \subfigure[$M_4$ \label{M4}]{\includegraphics[height=2.5cm]{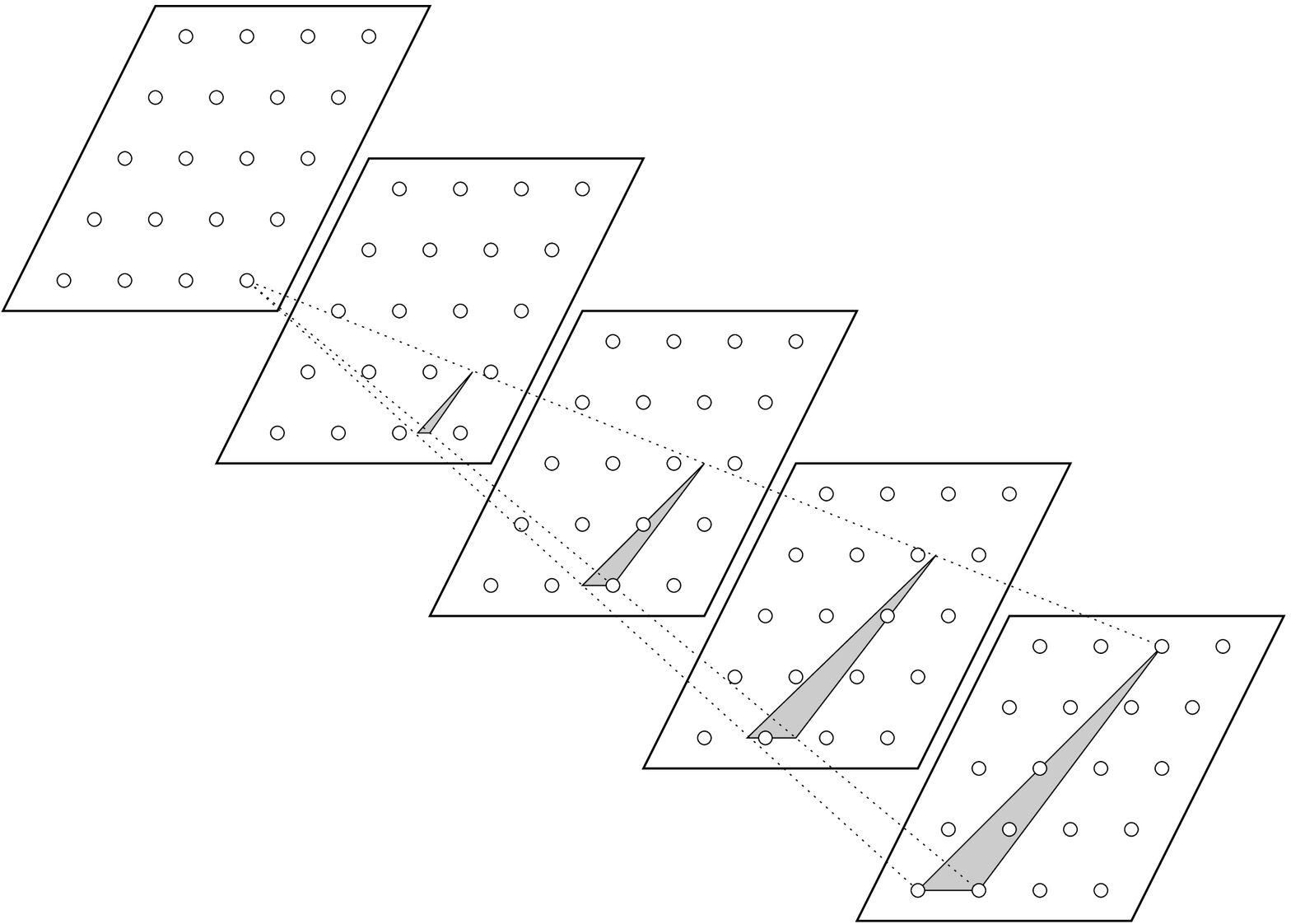}}
  %\\
  \qquad
  \subfigure[$M_5$ \label{M5}]{\includegraphics[height=2.5cm]{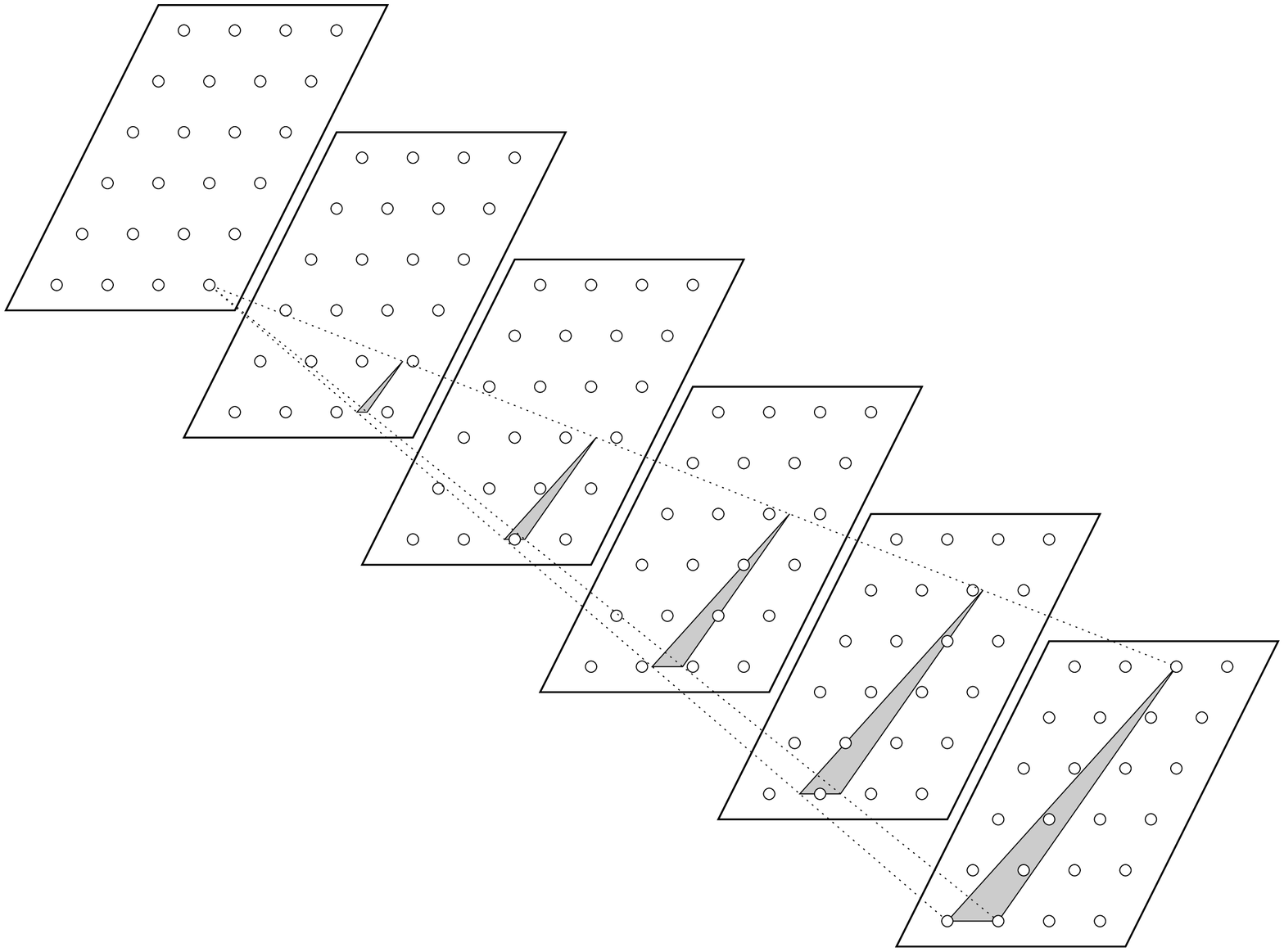}}
  \qquad
  \subfigure[$M_6$ \label{M6}]{\includegraphics[height=2.5cm]{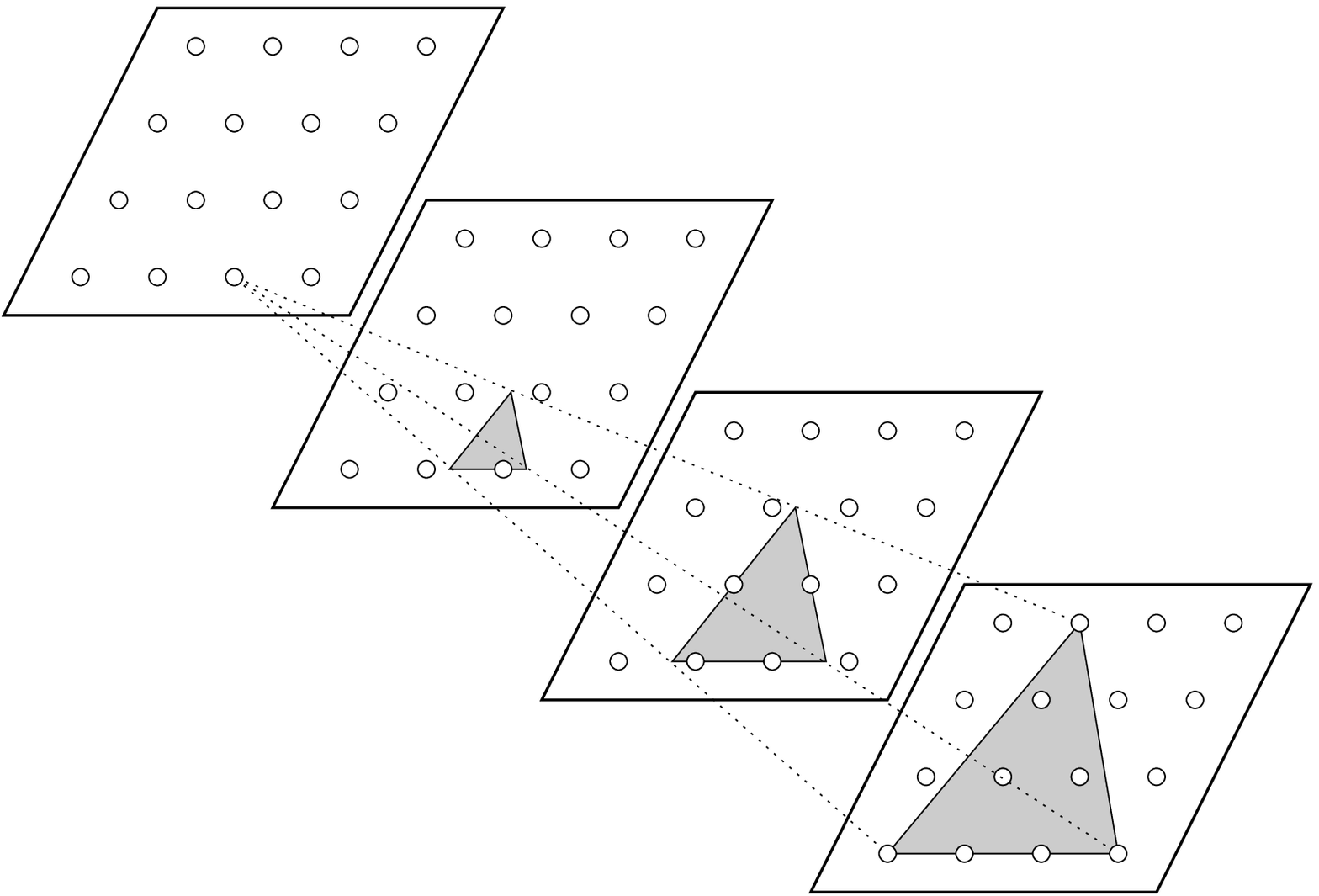}}
  \\
  %\qquad
  \subfigure[$M_7$ \label{M7}]{\includegraphics[height=2.5cm]{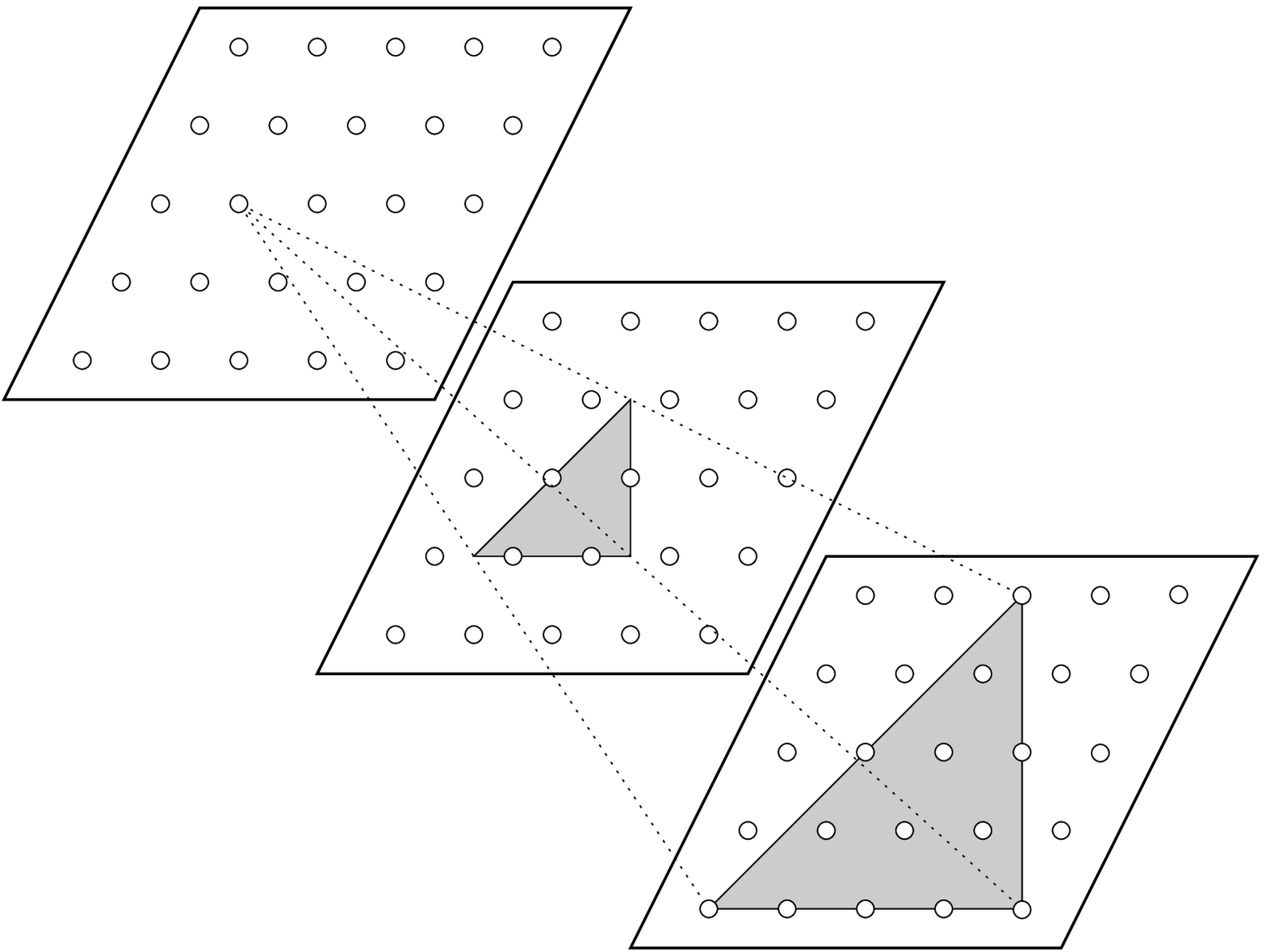}}
  \qquad
  \subfigure[$M_8$ \label{M8}]{\includegraphics[height=2.5cm]{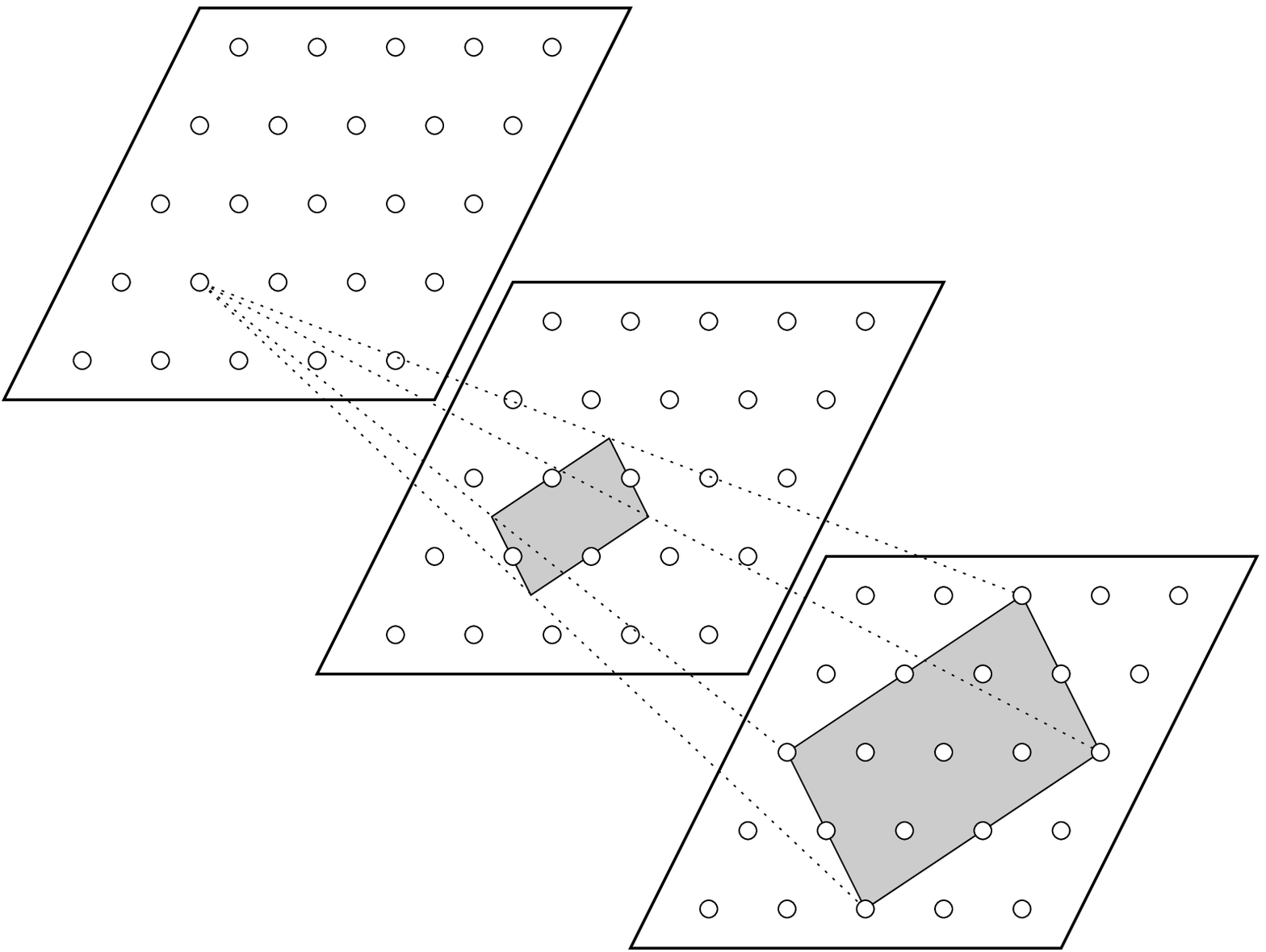}}
  %\\
  \qquad
  \subfigure[$M_9$ \label{M9}]{\includegraphics[height=2.5cm]{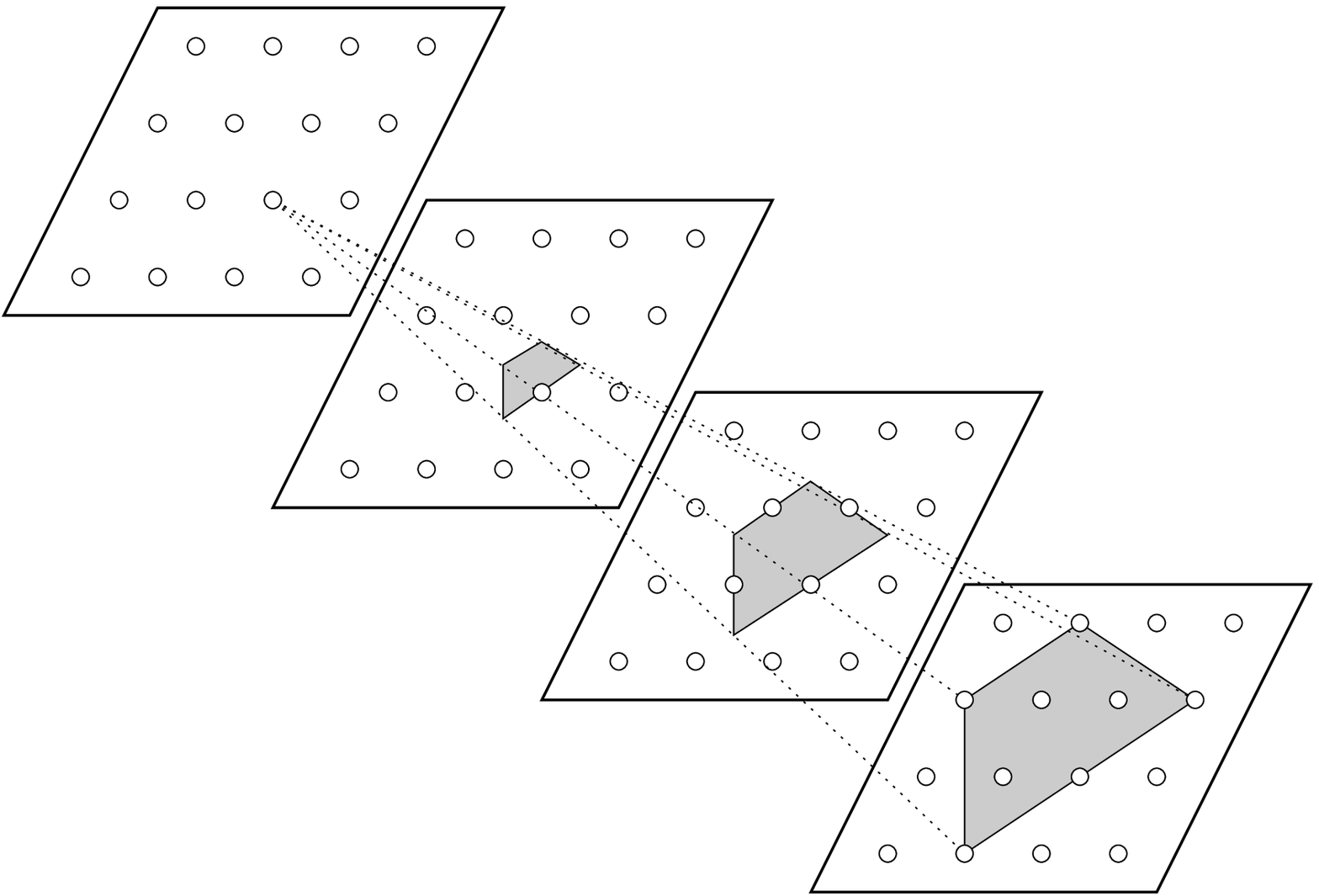}}
  \\
  %\qquad
  \subfigure[$M_{10}$ \label{M10}]{\includegraphics[height=2.5cm]{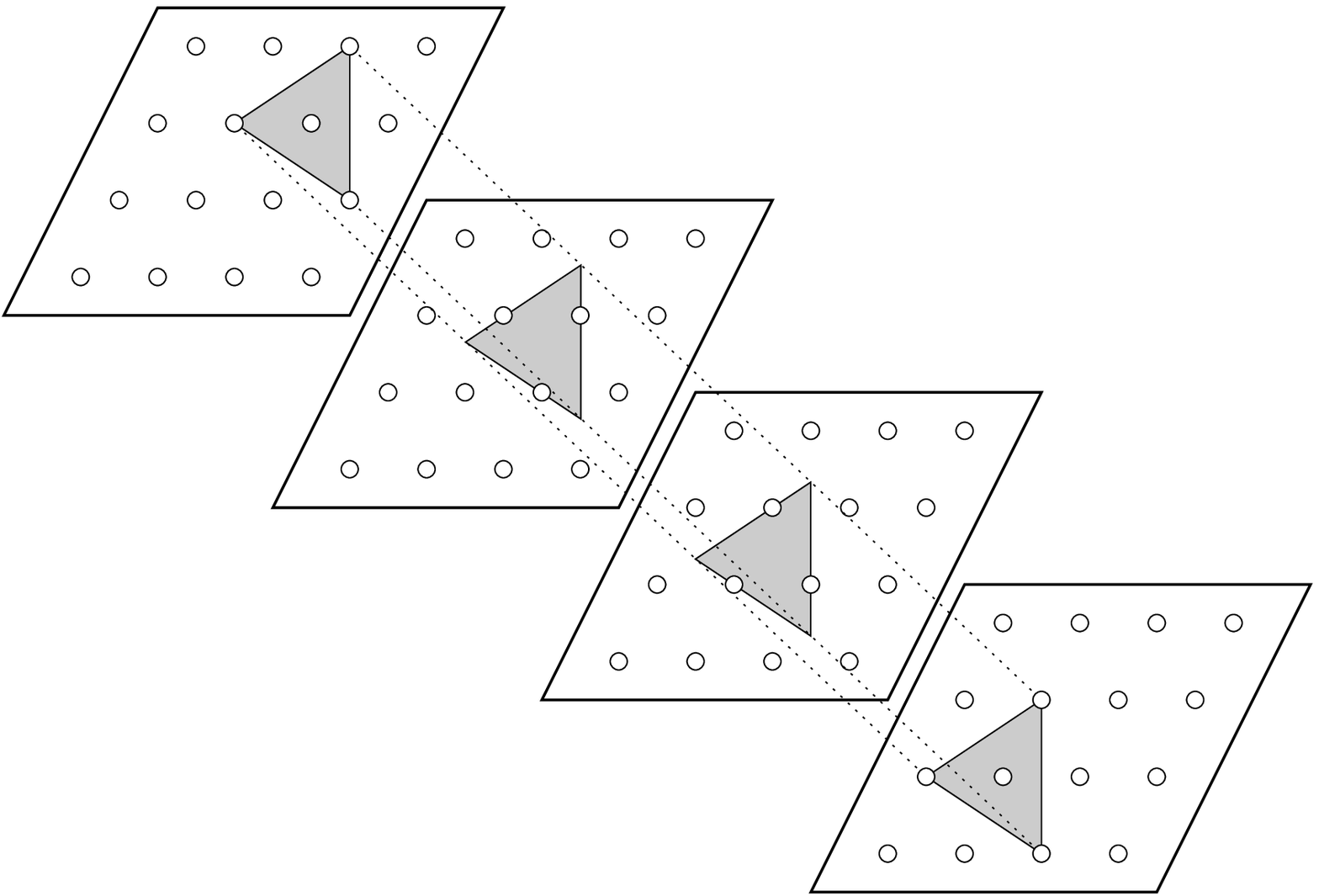}}
  %\\
  \qquad
  \subfigure[$M_{11}$ \label{M11}]{\includegraphics[height=2.5cm]{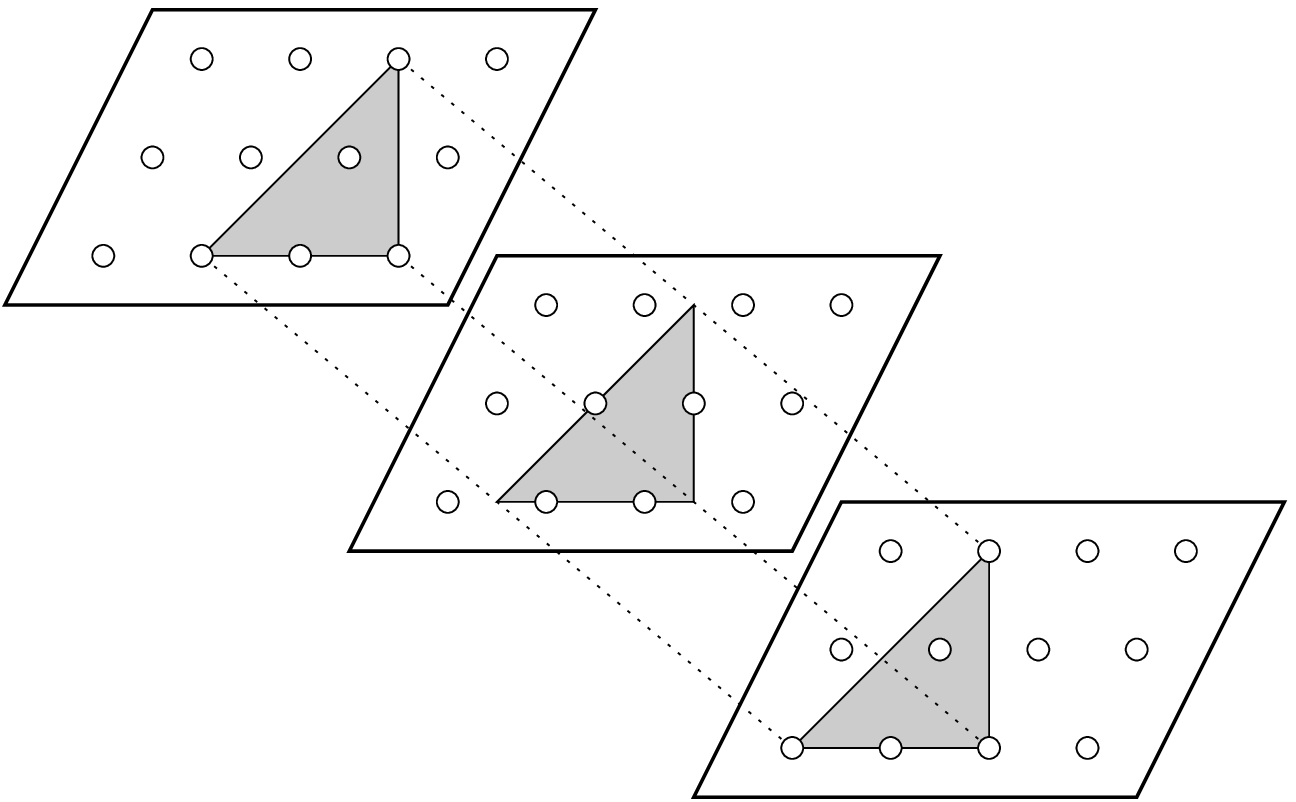}}
  \qquad
  \subfigure[$M_{12}$ \label{M12}]{\includegraphics[height=2.5cm]{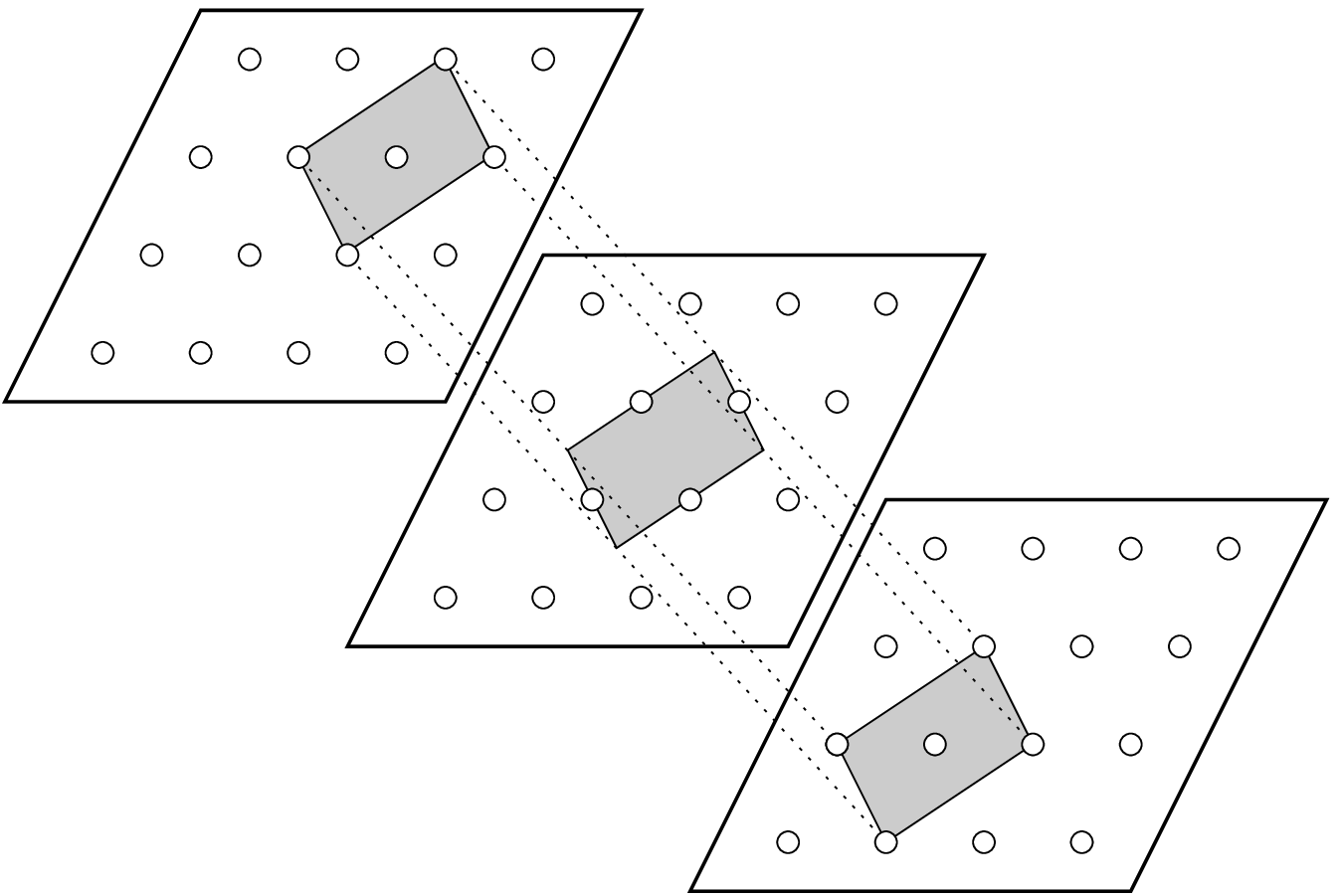}}
  \caption{All maximal lattice-free integral polytopes in dimension three}
  \label{fig.main.thm}
\end{figure}

\end{theorem}

We now introduce some further notation.
Throughout the paper, $d \in \natur$ is the dimension of the
underlying space.
Elements of $\R^d$ are considered to be column vectors.
Transposition is denoted by $(\cdot)^\top$ and the origin by
$o$.
We denote by $e_{j}$ the $j$th unit vector.
Its size will always be clear from the context.
For $x, y \in \real^d$, we denote by $[x,y]$ the line
segment with endpoints $x$ and $y$, and by $\ray{x,y}$ the
ray (i.e., the half-line) emanating from $x$ and passing
through $y$.
An additive subgroup $\Lambda$ of $\R^d$ is said to be a
\emph{lattice} if the intersection of $\Lambda$ with every
compact set of $\R^d$ is finite.
In this paper, for the sake of simplicity, we fix our
underlying lattice to be $\Z^d$, though, due to affine
invariance, the obtained results are independent of the
concrete choice of the lattice.

Given a set $K \subseteq \R^d$, we use the functionals
$\conv(K)$ (convex hull of $K$), $\aff(K)$ (affine hull of
$K$), $\lin(K)$ (linear hull of $K$), $\intr(K)$ (interior
of $K$), $\relintr(K)$ (relative interior of $K$),
$\relbd(K)$ (relative boundary of $K$), $\rec(K)$ (recession
cone of $K$), and $\vx(K)$ (set of vertices of $K$). 
For $K \subseteq \real^d$, $\vol(K)$ denotes the volume of
$K$ in $\aff(K)$.

The dual lattice of $\Lambda= s \integer^d$  is  $\Lambda^\ast =
\frac{1}{s} \Z^d$.
By $\pi$ we denote the projection onto the first $d-1$
coordinates, i.e., the mapping $\pi(x) := (x_1, \ldots,
x_{d-1})$, where $x := (x_1, \ldots, x_d) \in \real^d$.
This implies $\pi(\Lambda) = s\integer^{d-1}$.
If $K \subseteq \R^d$ is a closed convex set with nonempty
interior, then the \emph{lattice width} of $K$ (with respect
to the lattice $\Lambda$) is defined by
$$ w_\Lambda(K) : = \min_{u \in \Lambda^\ast \setminus \{o\}}
   w(K,u),$$ 
where $w(K,u)$, for $u \in \real^d$, is the
\emph{width function} given by 
$$w(K,u) := \max_{x \in K} u^\top x - \min_{x \in K} u^\top
x.$$
The lattice width of $K$ with respect to $\Lambda$ can be
seen as the smallest number of \textquotedblleft lattice
slices\textquotedblright~of $K$ along any nonzero vector in
$\Lambda^\ast$. 

Theorem \ref{finiteness} is proved in Section
\ref{sec.finiteness}.
In Section~\ref{notions}, we introduce the tools which we
need for proving Theorem \ref{main.thm} and we explain the idea
of the proof.
The proof of Theorem \ref{main.thm} is given in
Sections \ref{parity}--\ref{sect:computer-search}.

%%%%%%%%%%%%%%%%%%%%%%%%%%%%%%%%%%%%%%%%%%%%%%%%%%%%%%%%%%%%%%%%%%%%%%%
%% Finiteness %%
%%%%%%%%%%%%%%%%%%%%%%%%%%%%%%%%%%%%%%%%%%%%%%%%%%%%%%%%%%%%%%%%%%%%%%%

\section{The finiteness proof} \label{sec.finiteness}

In this section, we prove Theorem \ref{finiteness}.
Let us first highlight the main steps of the proof.

\begin{enumerate}
  \item {\bf Reduction to polytopes.}
    Every unbounded $P \in \cPifm{s}{d}$ is the direct
    product of an affine space and a polytope in
    $\cPifm{s}{k}$ for some $1 \le k \le d$ (see Proposition
    \ref{S=P+L}).
    Thus, it suffices to verify finiteness only for
    polytopes within $\cPifm{s}{d}$.
  \item {\bf Bounding $\boldsymbol{|P \cap \Lambda|}$.}
    Consider a polytope $P \in \cPifm{s}{d}$.
    We construct an upper bound on the number of points of
    $\Lambda$ on the boundary of $P$.
    For that, we use the lattice diameter.
    The \emph{lattice diameter} of $P$ with respect to
    $\Lambda$ is defined as the maximum of $|l \cap P \cap
    \Lambda|-1$ over all lines $l$. 
    We show that the lattice diameter of $P$ is bounded from
    above in terms of $d$ and $s$ only.
    This is done as follows.

    We assume by contradiction that, for some line $l$,  $|l
    \cap P \cap \Lambda|-1$ is a large number $M$.
    By a $\Lambda$-preserving transformation, $l =
    \lin(\{e_d\})$.
    Let $P'$ be the projection of $P$ onto the first $d-1$
    coordinates.
    Then $\pi(l) = o$ and from $P \in \cPifm{s}{d}$ it
    follows $\intr(P') \cap \pi(\Lambda) \ne \emptyset$ (see
    Lemma \ref{projection non-free}).
    Choose $p \in \intr(P') \cap \pi(\Lambda)$.
    Then we construct a $k$-dimensional simplex $S$ with
    vertices $ o=p_0, p_1, \dots, p_k$ in $\Z^{d-1}$
    such that $p$ is the only point of $\pi(\Lambda)$ in the
    relative interior of $S$.
    This construction is the key ingredient in our proof
    (see Lemma \ref{properties of tight}).
    Let $\lambda_0, \dots, \lambda_k$ be the barycentric
    coordinates of $p$ with respect to $S$.
    By results of \cite{hensley-1983,lagarias-ziegler-1991}
    (see Theorem \ref{hensley simplices with one point}),
    the $\lambda_i$'s are bounded from below in terms of $d$
    and $s$ only.
    The length of $(p+l) \cap P$ is bounded from below in
    terms of $\lambda_0$ and $M$.
    On the other hand, since $P$ is $\Lambda$-free, the
    length of $(p+l) \cap P$ is at most $s$.
    So, if $M$ is too large, this leads to a contradiction.

    The upper bound on the lattice diameter implies an upper
    bound on $|P \cap \Lambda|$ (see Lemma \ref{G bound}).
  \item {\bf Conclusion of finiteness.}
    The upper bound on $|P \cap \Lambda|$ together with
    results of \cite{hensley-1983,lagarias-ziegler-1991}
    (see Theorem \ref{hensley-finiteness}) implies an upper
    bound on the volume of $P$ (see Theorem
    \ref{finite-volume}).
    All bounds only depend on $d$ and $s$.
    This, in turn, yields finiteness of $\cPifm{s}{d}$ (see
    Theorem \ref{vol-bound<=>finiteness}).
\end{enumerate}

The fact that we can restrict to the study of polytopes in
$\cPifm{s}{d}$ is a consequence of the following
proposition.
We point out that a similar result is true for the set
$\cCfm{s}{d}$  as well (see
\cite[Proposition\,3.1]{Lovasz89}).

\begin{proposition} \label{S=P+L} \label{unbounded lat-free}
Let $d,s \in \natur$ and let $P \in \cPifm{s}{d}$.
Then there exists some $k \in \{1, \ldots, d\}$ and a
polytope $P' \in\cPifm{s}{k}$ such that $P \equiv P' \times
\real^{d-k} \modulo{\Aff(\Lambda)}$.
\end{proposition}

\begin{proof}
If $P$ is bounded, the assertion is trivial as we let
$k = d$ and $P' = P$.
Let $P$ be unbounded.
By an inductive argument, it suffices to show the existence
of $P' \in \cPifm{s}{d-1}$ such that $P \equiv P' \times
\real \modulo{\Aff(\Lambda)}$.

Since $P$ is unbounded, the recession cone of $P$ contains
nonzero vectors.
Since $P$ is integral, the recession cone of $P$ is an
integral polyhedron (see, for example,
\cite[\S16.2]{schrijver-book-86}).
Thus, the recession cone of $P$ contains a nonzero integer
vector $u$.
By scaling, we can assume that $u \in \Lambda$.
		
Applying a $\Lambda$-preserving transformation we assume
that $u = s e_d$.
The polyhedron $P' := \pi(P) \subseteq \real^{d-1}$ is
$\pi(\Lambda$)-free.
In fact, assume there exists a point $p' \in \intr(P') \cap
\pi(\Lambda)$.
Then $\intr(P) \cap \pi^{-1}(p')$ is nonempty and contains
infinitely many points of $\Lambda$, a contradiction to the
choice of $P$.

Since $P'$ is $\pi(\Lambda)$-free, $\pi^{-1}(P')$ is
$\Lambda$-free.
By construction, $P \subseteq \pi^{-1}(P')$, and since $P$
is maximal in $\cPif{s}{d}$ we even have $P=\pi^{-1}(P')$.
Furthermore, $P' \in \cPifm{s}{d-1}$.
In fact, if $P'$ were not maximal in $\cPif{s}{d-1}$ we
could find $P'' \in \cPif{s}{d-1}$ such that $P'
\varsubsetneq P''$.
Then $P$ is properly contained in the $\Lambda$-free
integral polyhedron $\pi^{-1}(P'')$, a contradiction to the
assumptions on $P$.
By construction, $P \equiv P' \times \real
\modulo{\Aff(\Lambda)}$.
\end{proof}

The following is well-known (see, for instance,
\cite[Theorem\,2]{lagarias-ziegler-1991}).

\begin{theorem} \label{vol-bound<=>finiteness}
Let $d,s \in \natur$ and let $\cX \subseteq \cPi^d$ be a set
of polytopes.
Then the set $\cX / \Aff(\Lambda)$ (where $\Lambda=s
\integer^d$)  is finite if and only if the volume of all
elements of $\cX$ is bounded from above by a constant
depending only on $d$ and $s$.
\end{theorem}

In the remainder of this section we prepare the proof of the
following theorem.

\begin{theorem} \label{finite-volume}
Let $d,s \in \natur$.
Then there exists a constant $V(d,s) > 0$ such that for every
polytope $P$ in $\cPifm{s}{d}$ the inequality $\vol(P) \le
V(d,s)$ is fulfilled.
\end{theorem}

Once, Theorem \ref{finite-volume} is proven, Theorem
\ref{finiteness} is a direct consequence of Proposition
\ref{unbounded lat-free}, Theorem
\ref{vol-bound<=>finiteness}, and Theorem
\ref{finite-volume}.

The proof of Theorem~\ref{finite-volume} relies on results
of Hensley \cite{hensley-1983} and Lagarias and Ziegler
\cite{lagarias-ziegler-1991}.
Hensley \cite{hensley-1983} showed that the volume and
the total number of integer points of a $d$-dimensional
integral polyhedron with precisely $k>0$ interior
integer points can be bounded in terms of $d$ and $k$ only.
Lagarias and Ziegler \cite{lagarias-ziegler-1991} improved
these bounds and generalized parts of Hensley's results.
In this paper, we shall use the main results as well as
some intermediate assertions from \cite{hensley-1983} and
\cite{lagarias-ziegler-1991}.

A polytope $S$ is said to be a \term{simplex} if $S$ is the
convex hull of finitely many affinely independent points.
If $S$ is a simplex in $\real^d$ with vertices
$p_0, \ldots, p_k$ ($0 \le k \le d$) and $p$ is a point in
$S$, then $p$ can be uniquely represented by
$p = \sum_{j=0}^k \lambda_j p_j$, where $\lambda_0, \ldots,
\lambda_k \ge 0$ and $\lambda_0 + \cdots + \lambda_k = 1$.
The values $\lambda_0, \ldots, \lambda_k$ are called the
\term{barycentric coordinates} of $p$ with respect to the
simplex $S$.
The point $p$ lies in the relative interior of $S$
if and only if $\lambda_0, \ldots, \lambda_k > 0$.

\begin{theorem}
  \thmcaption{\cite[Theorem~3.1]{hensley-1983} and
    \cite[Lemma\,2.2]{lagarias-ziegler-1991}} \label{hensley
    simplices with one point}
Let $d,s \in \natur$.
Then there exists a constant $\lambda^\ast(d,s) > 0$ such
that, for every $d$-dimensional integral simplex $S$ in
$\real^d$ with precisely one interior point $p$ in
$s\integer^d$, all barycentric coordinates $\lambda_i$
$(i=0,\ldots,d)$ of $p$ with respect to $S$ satisfy
$\lambda_i \ge \lambda^\ast(d,s)$.
\end{theorem}

Note that, in the formulation of Theorem \ref{hensley
simplices with one point}, $\lambda^\ast(d,s)$ is not
necessarily best possible.
Once some $\lambda^\ast(d,s)$ is known, then any smaller
positive constant works as well.
Thus, it is always possible and will be convenient later to
choose the values $\lambda^\ast(d,s)$ to be nonincreasing
in $d \in \natur$.
In fact, the best known concrete values for
$\lambda^\ast(d,s)$, given in
\cite[Lemma\,2.2]{lagarias-ziegler-1991}, are nonincreasing
in $d$.

\begin{theorem}
  \thmcaption{\cite[Theorem~3.6]{hensley-1983} and
    \cite[Theorem\,1]{lagarias-ziegler-1991}} \label{hensley-finiteness}
Let $d,s,k \in \natur$ and let $\Lambda = s\integer^d$.
Let $\cX$ be the class of all $d$-dimensional polytopes $P
\in \cPi^d$ with $1 \le |\intr(P) \cap \Lambda| \le k$.
Then there exists a constant $V(d,s,k) > 0$ such that for
every $P \in \cX$ one has $\vol(P) \le V(d,s,k)$.
In particular, $\cX / \Aff(\Lambda)$ is a finite set.
\end{theorem}

%%%%%%%%%%%%%%%%%%%%%%%%%%%%%%%%%%%%%%%%%%%%%%%%%%%%%%%%%%%%%%%%%%%%%%%%
%% Notation and Proof of the Key Lemma %%
%%%%%%%%%%%%%%%%%%%%%%%%%%%%%%%%%%%%%%%%%%%%%%%%%%%%%%%%%%%%%%%%%%%%%%%%

We have mentioned all results from the literature that are
needed to prove Theorem \ref{finite-volume}.
Let us now show our assertion.
We point out that in the remainder of this section, for all
statements and proofs, we always assume $\Lambda = s \Z^d$.

Let $a \in \Lambda$ and let $\cX^d(a)$ be the class of
all polyhedra $P \in \cPi^d$ such that $a \in \relbd(P)$ and
$\relintr(P) \cap \Lambda \ne \emptyset$.
On $\cX^d(a)$ we introduce the partial order $\preceq$ as
follows: for $P, Q \in \cX^d(a)$ we define $P \preceq Q$ if
and only if $\relintr(P) \subseteq \relintr(Q)$.
Let us verify that the binary relation $\preceq$ is indeed a
partial order.
The property $P \preceq P$ is obvious.
If $P \preceq Q$ and $Q \preceq P$, then $\relintr(P) =
\relintr(Q)$.
Since $P$ and $Q$ are closed it follows $P = Q$.
If $P \preceq Q$ and $Q \preceq R$, then $\relintr(P)
\subseteq \relintr(Q) \subseteq \relintr(R)$.
Thus $P \preceq R$.

By $\cR^d(a)$ we denote the set of the minimal elements of
the poset $(\cX^d(a),\preceq)$, i.e., the set of the
elements $Q \in \cX^d(a)$ such that there exists no $P \in
\cX^d(a)$ with  $P \preceq Q$ and $P \ne Q$.
We emphasize that elements of $\cX^d(a)$ and $\cR^d(a)$ do
not have to be full-dimensional.
It is not hard to verify that for every $P \in \cX^d(a)$
there exists $Q \in \cR^d(a)$ such that $Q \preceq P$.
If $P$ is bounded, this follows from the fact that the set
of all $Q \in \cX^d(a)$ satisfying $Q \preceq P$ is finite
as $\card{P \cap \integer^d} < +\infty$.
If $P$ is unbounded we replace $P$ by $\bar{P} = \conv(P
\cap B \cap \integer^d)$, where $B$ is a sufficiently large
box centered at $a$ and such that $\relintr(\bar{P}) \cap
\Lambda \ne \emptyset$.
Then we apply the argument for the bounded case to
$\bar{P}$.
We remark that for $P, Q \in \cX^d(a)$ the condition
$\relintr(P) \subseteq \relintr(Q)$ holds if and only if one
has $P \subseteq Q$ and $\relintr(P) \cap \relintr(Q) \ne
\emptyset$.
This follows from standard results in convexity (see, for
example, \cite[Theorem\,6.5]{Rockafellar72}).

It turns out that the elements of $\cR^d(a)$ have a very
specific shape which is described as follows.

\begin{lemma} \label{properties of tight}
Let $a \in \Lambda$ and $P \in \cR^d(a)$.
Then $P$ has the following properties.
  \begin{enumerate}[I.]
    \item \label{a-red is simplex} $P$ is a simplex of
      dimension $k \in \{1,\ldots,d\}$.
    \item \label{a-red a is vertex} $a \in \vx(P)$.
    \item \label{a-red only one int}$\relintr(P) \cap
      \Lambda$ consists of precisely one point.
    \item \label{a-red opp is empty} The facet $F$ of $P$
      opposite to the vertex $a$ satisfies $F \cap
      \integer^d = \vx(F)$.
  \end{enumerate}
\end{lemma}

\begin{proof}
Let $P \in \cR^d(a)$ and $q \in \relintr(P) \cap \Lambda$
be arbitrary.
We consider $2 q -a$ (the reflection of $a$ with respect to
$q$).
First assume that $2q - a \in P$.
Then $q \in \relintr(P) \cap \relintr([a,2q-a])$ and
$[a,2q-a] \subseteq P$.
Thus, since $P \in \cR^d(a)$ we have $P=[a,2q-a]$.
Again, since $P \in \cR^d(a)$, $q$ is the only point of
$\Lambda$ in $\relintr(P)$.
For such a $P$,
Parts~\ref{a-red is simplex}--\ref{a-red opp is empty}
follow immediately.
In the remainder of the proof let $2q- a \not \in P$.

\emph{Parts~\ref{a-red is simplex} and \ref{a-red a is
    vertex}.}
Let $b$ be the intersection point of $\ray{a,q}$ and
$\relbd(P)$.
Since $q \in \relintr([a,b])$ we have $q=(1-\lambda) a +
\lambda b$ for some $0 < \lambda < 1$.
Consider a facet $F$ of $P$ which contains $b$.
Since $P$ is integral, also $F$ is integral, i.e.,
$F = \conv (F \cap \integer^d)$.
By Carath\'eodory's theorem, there exist affinely
independent points $q_1, \ldots, q_k \in F \cap \integer^d$
such that $b = \lambda_1 q_1 + \cdots + \lambda_k q_k$ for
some $\lambda_1,\ldots,\lambda_k > 0$ with $\lambda_1 +
\cdots + \lambda_k = 1$.
Thus, $q = (1-\lambda) q_0 + \lambda \lambda_1 q_1 \cdots +
\lambda \lambda_k q_k$, where $q_0 := a$.
The point $a = q_0$ does not belong to $\aff(F)$.
In fact, otherwise $a \in P \cap \aff(F) = F$ and since $b
\in F$ we get $q \in F$, a contradiction to $q \in
\relintr(P)$.
Hence $q_0,\ldots,q_k$ are affinely independent.
Since $P \in \cR^d(a)$, we have
$P = \conv(\{q_0,\ldots,q_k\})$.
Hence $P$ is a simplex and $a$ is a vertex of $P$.

In the remainder of the proof let $P = \conv(\{q_0, \ldots,
q_k\})$ with $q_0 := a$ and $q_1, \ldots, q_k$ defined as
above.

\emph{Part~\ref{a-red only one int}.}
For $j=0,\ldots,k$ let $P_j$ be the simplex with vertices
$\{q,q_0,\ldots,q_k\} \setminus \{q_j\}$.
It can be verified with straightforward arguments that
$P = P_0 \cup \cdots \cup P_k$ and the relative interiors of
the simplices $P_j$ are pairwise disjoint.
For proving Part~\ref{a-red only one int}, we argue by
contradiction.
We assume that $\relintr(P) \cap \Lambda$ contains $q'$ with
$q' \ne q$.
First we show that $q' \in P_0$.
Assume the contrary.
Then $q' \in P_j$ for some $j \in \{1, \ldots, k\}$.
Let $F$ be the face of $P_j$ with $q' \in \relintr(F)$.
Since $q' \not\in P_0$, $a$ is a vertex of $F$.
The existence of $F \varsubsetneq P$ with $q' \in
\relintr(F)$ and $a \in \vx(F)$ contradicts the fact that $P
\in \cR^d(a)$.
Hence $q' \in P_0$.
We define $Q:= \conv ((P_0 \cap \integer^d) \setminus
\{q\})$.
Since $q' \in \relintr(P)$, and $q', q_1,\ldots,q_k \in Q$,
the polytope $Q$ has the same dimension as $P$.
We have $\ray{a,q} \cap Q = [b,b']$, where $b \in
\relintr(P)$ and $b' \in \relbd(P)$.
Since $q \in \relintr([a,b])$ one has $q = (1-\lambda)a +
\lambda b$ for some $0 < \lambda < 1$.
Let now $G$ be the facet of $Q$ containing $b$.
The point $a=q_0$ does not belong to $\aff(G)$.
In fact, otherwise $\aff(G)$ would contain $[b,b']$, which
implies $\aff (G) \cap \relintr Q \ne \emptyset$, a
contradiction.
Using Carath\'eodory's theorem, let $p_1,\ldots,p_m$ be
affinely independent vertices of $G$ such that $b=\lambda_1
p_1 + \cdots + \lambda_m p_m$ for some
$\lambda_1,\ldots,\lambda_m > 0$ with $\lambda_1 + \cdots +
\lambda_m = 1$.
Then $q = (1-\lambda)p_0 + \lambda \lambda_1 p_1 + \cdots +
\lambda \lambda_m p_m$ with $p_0 := a$.
Since $p_0 \not \in \aff(G)$ and since $p_1, \ldots, p_m \in
G$ are affinely independent, we see that $p_0, \ldots, p_m$
are affinely independent.
The simplex $S = \conv(\{p_0, \ldots, p_m\})$ is properly
contained in $P$, contains $a$ on its relative boundary and
satisfies $q \in \relintr(S) \cap \relintr(P)$, a
contradiction to the fact that $P \in \cR^d(a)$.
This shows Part~\ref{a-red only one int}.

\emph{Part~\ref{a-red opp is empty}}.
We argue by contradiction.
Let $F$ be the facet of $P$ opposite to $a$ and assume that
$\vx(F) \varsubsetneq F \cap \integer^d$.
Let $S_1, \ldots, S_m$ be a triangulation constructed on the
points $F \cap \integer^d$.
Then, $S_1, \ldots, S_m$ are simplices with pairwise
disjoint interiors having the same dimension as $F$ and such
that $F \cap \integer^d = \bigcup_{i=1}^m \vx(S_i)$, $F =
\bigcup_{i=1}^m S_i$, and for every $S_i$, $\vx(S_i)$ are the
only integer points in $S_i$.
By assumption, we have $S_i \ne F$ for every $i$.

Then there exists a simplex $S_j$ such that $\ray{a,q} \cap
S_j$ is nonempty.
Let $b$ be the point $\ray{a,q} \cap S_j$.
Further on, let $G$ be the face of $S_j$ with $b \in
\relintr(G)$.
By construction, $q \in \relintr(\bar{P})$ where $\bar{P} :=
\conv (\{a\} \cup G)$ and $\bar{P} \varsubsetneq P$.
This contradicts the fact that $P \in \cR^d(a)$.
\end{proof}

Lemma \ref{properties of tight} and the following Lemma
\ref{projection non-free} are used in the proof of Lemma
\ref{G bound}.

\begin{lemma} \label{projection non-free}
Let $P \in \cPifm{s}{d}$ be a polytope.
Then $\intr(\pi(P)) \cap \pi(\Lambda) \ne \emptyset$.
\end{lemma}

\begin{proof}
If $P' := \pi(P)$ satisfies $\intr(P') \cap \pi(\Lambda) =
\emptyset$, then $\pi^{-1}(P')$ is $\Lambda$-free and
integral, and then in view of the maximality of $P$, one has
$\pi^{-1}(P') \subseteq P$ which contradicts the boundedness
of $P$.
\end{proof}

In the following lemma we prove that the number of points of
$\Lambda$ on the boundary of a polytope $P \in \cPifm{s}{d}$
is bounded by a constant which is dependent only on $d$ and
$s$.

\begin{lemma} \label{G bound}
Let $d,s \in \natur$.
Then there exists a constant $N(d,s) > 0$ such that every
polytope $P \in \cPifm{s}{d}$ contains at most $N(d,s)$
points in $\Lambda$.
\end{lemma}

\begin{proof}
Let $P \in \cPifm{s}{d}$ be a polytope.
For the purpose of deriving a contradiction assume that
$|P \cap \Lambda| \ge M^d + 1$, where
$M = \ceiling{\frac{1}{\lambda^\ast(d,s)} + 1}$ with
$\lambda^\ast(d,s)$ defined as in the formulation of
Theorem~\ref{hensley simplices with one point}.
Thus, there exist two distinct points $v, w \in
P \cap \Lambda$ such that $\frac{1}{s} v \equiv \frac{1}{s}
w \modulo{M}$.
Then we can choose pairwise distinct $z_0, \ldots, z_M$ in
$P \cap \Lambda \cap \aff(\{v,w\})$ such that
$\conv(\{z_0, \ldots, z_M\}) \cap \Lambda = \{z_0, \ldots,
z_M\}$.
Performing a $\Lambda$-preserving transformation we assume
that $z_j = j \cdot s e_d$ for $j = 0, \ldots, M$.
One has $\pi(z_j) = o$ for every $j = 0, \ldots, M$.
Since $M \ge 2$ (which follows from $\lambda^\ast(d,s) > 0$), 
$o$ is a boundary point of $P' := \pi(P)$,
otherwise $P$ would not be $\Lambda$-free.
By Lemma~\ref{projection non-free}, $\intr(P') \cap
\pi(\Lambda) \ne \emptyset$.

By construction, $P'$ is integral and belongs to
$\cX^{d-1}(o)$.
Thus, there exists a polytope $Q \in \cR^{d-1}(o)$ with
$\relintr(Q) \subseteq \intr(P')$.
By Lemma~\ref{properties of tight}, $Q$ is a simplex with
precisely one point of $\pi(\Lambda)$, say $p$, in the
relative interior.
Let $k$ be the dimension of $Q$ and let $p_0,\ldots,p_k$ be
the vertices of $Q$ with $p_0 = o$.
By Theorem~\ref{hensley simplices with one point}, if
$p = \sum_{j=0}^k \lambda_j p_j$ with
$\lambda_0,\ldots,\lambda_k >0$ and $\lambda_0 + \cdots +
\lambda_k = 1$, then one has $\lambda_j \ge
\lambda^\ast(d,s)$ for every $j=0,\ldots,k$, where
$\lambda^\ast(d,s)$ is a constant as in the formulation of
Theorem \ref{hensley simplices with one point}.
For a point $x \in P'$, let $f(x)$ denote the length of
the line segment $\pi^{-1}(x) \cap P$ (and thus, $f$
represents an ``X-Ray picture'' of $P$).
Employing the convexity of $P$ we see that $f(\cdot)$ is
concave on $P'$.
Consequently,
\begin{align*}
  f(p) = f\left(\sum_{j=0}^k \lambda_j p_j \right) \ge
  \sum_{j=0}^k \lambda_j f(p_j) \ge \lambda_0 f(p_0) \ge
  \lambda^\ast(d,s) s M > s.
\end{align*}
On the other hand, since $p \in \intt(P') \cap
\pi(\Lambda)$, we have $f(p) \le s$ as otherwise $P$ would
not be $\Lambda$-free.
Thus, this gives a contradiction to our assumption on $|P
\cap \Lambda|$.
It follows that $P$ contains at most $M^d$ points in
$\Lambda$ and we can choose $N(d,s) := M^d$.
\end{proof}

\begin{proof}[Proof of Theorem~\ref{finite-volume}]
Let $P \in \cPifm{s}{d}$ be a polytope.
In the following, we enlarge $P$ to a polytope $Q \in
\cPi^d$ such that $P \subseteq Q$ and $\emptyset \neq
\intt(Q) \cap \Lambda \subseteq P \cap \Lambda$.
By Lemma~\ref{G bound}, this implies $1 \le |\intt(Q) \cap
\Lambda| \le |P \cap \Lambda| \leq N(d,s)$, with $N(d,s)$
defined in the formulation of Lemma~\ref{G bound}.
Then, by Theorem~\ref{hensley-finiteness}, $\vol(P) \le
\vol(Q) \le V(d,s,N(d,s))$ with $V(d,s,N(d,s))$ defined
according to Theorem~\ref{hensley-finiteness}.
Consequently, $\vol(P) \le V(d,s) := V(d,s,N(d,s))$.

Let us now construct the polytope $Q$.
For that, we consider a sequence of polytopes $P^i$ which we
define iteratively.
Choose an arbitrary $p_1 \in \Lambda$ such that $p_1 \not
\in P$ and let $P^1 := \conv (P \cup \{p_1\})$.
For $i \ge 1$, we proceed as follows.
If $\intt(P^i) \cap \Lambda \subseteq P \cap \Lambda$, then
we stop and define $Q := P^i$.
Otherwise, we select $p_{i+1} \in (\intt(P^i) \cap \Lambda)
\setminus (P \cap \Lambda)$ and set $P^{i+1} := \conv(P \cup
\{p_{i+1}\})$.
Note that $P^{i+1} \varsubsetneq P^i$ for all $i \ge 1$ and
that the sequence is finite since $P$ is a polytope.
Eventually, we construct a polytope $Q \in \cPi^d$ such that
$P \subseteq Q$ and $\intt(Q) \cap \Lambda \subseteq P \cap
\Lambda$.
Furthermore, $\intt(Q) \cap \Lambda \neq \emptyset$ since
$P$ is properly contained in $Q$ and $P$ is maximal
$\Lambda$-free.
\end{proof}

\begin{proof}[Proof of Theorem \ref{finiteness}]
The theorem follows directly from
Proposition \ref{unbounded lat-free},
Theorem \ref{vol-bound<=>finiteness}, and
Theorem \ref{finite-volume}.
\end{proof}

\begin{remark} \label{study reduced} \thmcaption{The role of
    $\cR^d(a)$}
In our proofs we use the class $\cR^d(a)$.
The properties of elements of this class are stated in
Lemma~\ref{properties of tight}.
It seems that the class $\cR^d(a)$ deserves an independent
consideration.
\end{remark}

\begin{remark} \thmcaption{Growth of constants}
Let us analyze the growth of the constants in our
statements.

$V(d,s,k)$ must be (at least) double exponential in $d$.
It can be chosen to be $k (8ds)^d (8s+7)^{d 2^{2d+1}}$ (see
\cite{Pikhurko01}).

From the proof of Theorem \ref{finite-volume} and the above
bound on $V(d,s,k)$, it follows that for a polytope $P \in
\cPifm{s}{d}$ we have $\vol(P) \le V(d,s,N(d,s)) \le N(d,s)
\cdot (8ds)^d (8s+7)^{d 2^{2d+1}}$.
By the proof of Lemma \ref{G bound}, we can choose
$N(d,s) = \ceiling{\frac{1}{\lambda^\ast(d,s)}+1}^d$.

The best known lower bound on the constant
$\lambda^\ast(d,s)$ is $(7(s+1))^{-2^{d+1}}$ (see
\cite[Lemma\,2.1]{lagarias-ziegler-1991}).
Substituting this into the above formula yields
\begin{equation} \label{vol P upper}
\vol(P) \le \ceiling{(1 + \big(7(s+1)\big)^{2^{d+1}}}^d
(8ds)^d (8s+7)^{d 2^{2d+1}}.
\end{equation}
In the asymptotic notation the bound can be expressed as
$\vol(P) = (s+1)^{O(d 4^d)}$.

Below we give an example which shows that the maximum volume
over all polytopes of $\cPifm{s}{d}$ is at least of order
$(s+1)^{\Omega(2^d)}$.
We use the following sequence considered in
\cite[Lemma~2.1]{lagarias-ziegler-1991}.
(For the sake of simplicity the dependency on $s$ is not
indicated explicitly.)
Let us define $y_1 := s + 1$ and $y_j := 1 + s
\prod_{i=1}^{j-1} y_i$ for $j \ge 2$ (equivalently one can
use the recurrency $y_j = y_{j-1}^2 - y_j + 1$).
For every $d \ge 2$, we introduce the simplex $S_d := \conv
(\{ y_1 e_1, \ldots,y_{d-1} e_{d-1}, (y_d-1) e_d\})$.
The verification of the fact that $S_d$ belongs to
$\cPifm{s}{d}$ (and even to $\cPfmi{s}{d}$) is left to the
reader.
The volume of $S_d$ can be expressed by
\begin{equation*}
  \vol(S_d) = \frac{1}{d!} \left(\prod_{i=1}^{d-1}
    y_i\right) (y_d-1) = \frac{1}{d!} \frac{1}{s}
  (y_d-1)^2.
\end{equation*}
As was noticed in \cite[p.~1026]{lagarias-ziegler-1991} one
has $y_d \ge (s+1)^{2^{d-2}}$ for all $d \ge 2$. This shows
$\vol(S_d) = (s+1)^{\Omega(2^d)}$.

\end{remark}

Bound \eqref{vol P upper} does not help to determine all bounded elements of $\cPifm{s}{d}$ for fixed values of $d$ and $s$ since the right hand side is tremendously large (for example, more than $10^{500}$ for $d = 3$ and $s = 1$). This is the reason, why our proof of Theorem~\ref{main.thm} (presented in the following sections) does not rely on \eqref{vol P upper}.

%%%%%%%%%%%%%%%%%%%%%%%%%%%%%%%%%%%%%%%%%%%%%%%%%%%%%%%%%%%%%%%%%%%%%%%
%% Notation for the Explicit Description in 3D %%
%%%%%%%%%%%%%%%%%%%%%%%%%%%%%%%%%%%%%%%%%%%%%%%%%%%%%%%%%%%%%%%%%%%%%%%

\section{Tools and proof outline for the explicit description in\\ 
  dimension three} \label{notions}

In Sections \ref{parity},
\ref{slicing}, and \ref{four.facets} we use the following additional notation. The area of a set $K \subseteq \R^2$ is denoted $\area(K)$
(which is shorter than $\vol(K)$ which we used in the
previous sections). 
Since the only lattice in Sections \ref{parity},
\ref{slicing}, and \ref{four.facets} is the standard
lattice, we write $w(K)$ instead of
$w_\Lambda(K)$ to denote the
lattice width. In this
paper, a \emph{polygon} is a two-dimensional polytope.
If $P$ is a polygon with integer vertices we use $i(P)$ and
$b(P)$ to denote the number of integer points in the
interior and on the boundary of $P$, respectively.  The well-known \term{formula of Pick} states that $\area(P) = i(P)
+ \frac{b(P)}{2} - 1$. 

Let us explain the structure of the proof of Theorem
\ref{main.thm} and introduce the tools used in the proof.

The proof is essentially based on the following two ideas. 
We use the parity argument (a rather common tool in the geometry of numbers).
Two integer points $x,y \in \Z^d$ are said to have the same
\emph{parity} if each component of $x-y$ is even, i.e., $x
\equiv y \modulo{2}$.
It is easily seen that the point $\frac{1}{2}(x+y)$ is
integer if and only if $x$ and $y$ have the same parity.
We will apply this argument to integer points on the
boundary of $P \in \M$ by exploiting the fact that each
facet of $P$ contains an integer point in its relative
interior (which is a property of maximal lattice-free convex
sets, see Lov\'{a}sz \cite{Lovasz89}).
Clearly, there are at most $2^3 = 8$ points of different
parity in dimension three.
Proofs based on this idea are presented in
Section~\ref{parity}.

The second idea is the following.
We take an arbitrary facet $F$ of $P \in \M$ and assume
without loss of generality that $F \subseteq \R^2 \times
\{0\}$ and $P \subseteq \R^2 \times \R_{\ge 0}$.
Then, we consider the section $F' = P \cap (\R^2 \times
\{1\})$.
Taking into account that $F$ is an integral polygon and
contains at least one integer point in its relative interior
and that $F'$ is lattice-free in $\R^2 \times \{1\}$ with
respect to the lattice $\Z^2 \times \{1\}$. It follows that
either $P$ is \textquotedblleft not too
high\textquotedblright~with respect to $F$ or that $F$
contains a bounded number of integer points.
Proofs based on this idea are presented in
Sections~\ref{slicing} and \ref{four.facets}.

\newcommand{\cF}{\mathcal{F}}
Let $P, Q$ be polytopes and let $\cF(P)$ and $\cF(Q)$ be the sets of all faces of $P$ and $Q$, respectively. Then $P$ and $Q$ are said to be
\emph{combinatorially equivalent} (or of the same \term{combinatorial type}) if there exists a bijection $T : \cF(P) \rightarrow \cF(Q)$ satisfying $T(F_1) \subseteq T(F_2)$ for all $F_1, F_2 \in \cF(P)$ with $F_1 \subseteq F_2$. Our first lemma dealing with $\M$ shows that every element of $\M$ has at most six facets. This yields a quite short list of possible combinatorial types for elements in $\M$. Our arguments proceed by distinction of different possible combinatorial types. The description of $P \in \M$ with six facets resp.~five facets is given 
in Sections \ref{parity} resp.~\ref{slicing}. The description of $P \in \M$ with four facets (i.e., of simplices in $\M$) can be found in \cite{anwawe}. Since the arguments in \cite{anwawe} are very lengthy, we present an alternative shorter analysis in Section \ref{four.facets}.
The following classes of polytopes will be relevant. 

\begin{itemize} 
  \item A polytope $P \subseteq\R^3$ is said to be a
    \emph{pyramid} if $P = \conv (F \cup \{a\})$, where $F$
    is a polygon and $a \in \R^3 \setminus \aff(F)$.
    In this case $F$ is called the \emph{base} and $a$ the
    \emph{apex} of $P$.
  \item A polytope $P \subseteq \R^3$ is said to be a
    \emph{prism} if $P=F+I$, where $F$ is a polygon and $I$
    is a segment which is not parallel to $F$.
    In this case $F+v$ with $v \in \vertt(I)$ are called the
    \emph{bases} of $P$.
  \item A polytope $P \subseteq \R^3$ is said to be a
    \emph{parallelepiped} if $P=I_1+I_2+I_3$ where $I_1,
    I_2, I_3$ are segments whose directions form a basis of
    $\R^3$.
\end{itemize}

In the rest of this section we present results which we  use as tools. 
From \cite{AverkovWagner,Hurkens90} a relation between area
and lattice width is known.
In the following theorem, \eqref{lwidth.bound} is shown in
\cite{Hurkens90} and \eqref{area vs small width} and
\eqref{area vs big width} in \cite{AverkovWagner}.

\begin{theorem} \label{area,width}
Let $K \subseteq \R^2$ be a lattice-free closed convex set
with $w := w(K) > 1$.
Then
\begin{align}
  w   & \le 1+\frac{2}{\sqrt{3}}, & & \label{lwidth.bound} \\
  \area(K) & \le \frac{w^2}{2(w-1)} & &
             \mbox{ if } 1 < w \le 2, \label{area vs small width} \\
  \area(K) & \le 2, & & \mbox{ if } 2 < w \le
             1+\frac{2}{\sqrt{3}}. \label{area vs big width}
\end{align}
\end{theorem}

The bound \eqref{area vs big width} is not sharp but
sufficient for our purposes (for the sharp upper bound see
\cite[Theorem\,2.2]{AverkovWagner}).
For $h \in \Z$ the set $\Z^2 \times \{h\}$ in the affine
space $\R^2 \times \{h\}$ can be naturally identified with
the lattice $\Z^2$ in $\R^2$.
Such identification will be used several times.

For characterizing faces of maximal lattice-free polytopes we need results on the description of polygons with a given small number of interior integer points. In particular, 
we need the following result of Rabinowitz
\cite{Rabinowitz89}. 

\begin{theorem} \thmcaption{\cite{Rabinowitz89}} \label{i=1:enum}
Let $P \subseteq \R^2$ be an integral polygon with exactly
one interior integer point.
Then $P$ is equivalent to one of the polygons shown in
Figure~\ref{i=1}.
\begin{figure}[ht]
  \centering
  \subfigure[\label{tria1-1}]{\includegraphics[height=1.5cm]{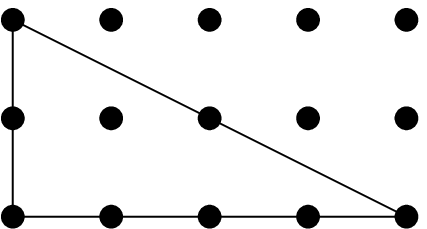}}
  \qquad
  \subfigure[\label{tria1-2}]{\includegraphics[height=1.5cm]{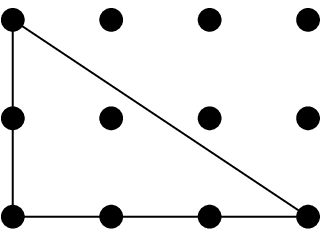}}
  \qquad
  \subfigure[\label{tria1-3}]{\includegraphics[height=1.5cm]{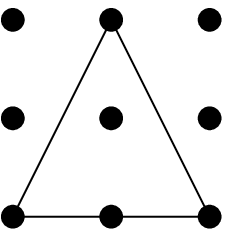}}
  \qquad
  \subfigure[\label{tria1-4}]{\includegraphics[height=2.25cm]{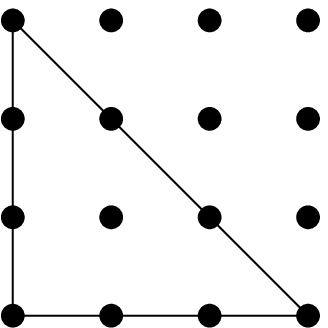}} \\
  %\qquad
  \subfigure[\label{tria1-5}]{\includegraphics[height=1.5cm]{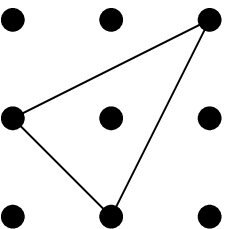}}
  \qquad
  \subfigure[\label{quad1-1}]{\includegraphics[height=1.5cm]{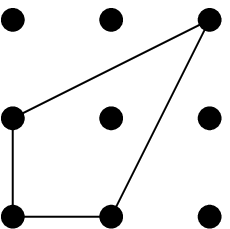}}
  \qquad
  \subfigure[\label{quad1-2}]{\includegraphics[height=1.5cm]{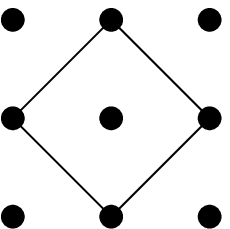}}
  \qquad
  \subfigure[\label{quad1-3}]{\includegraphics[height=1.5cm]{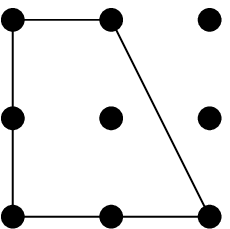}} \\
  %\qquad
  \subfigure[\label{quad1-4}]{\includegraphics[height=1.5cm]{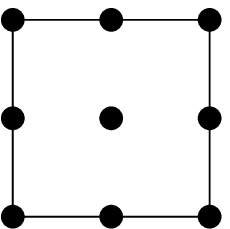}}
  \qquad
  \subfigure[\label{quad1-5}]{\includegraphics[height=1.5cm]{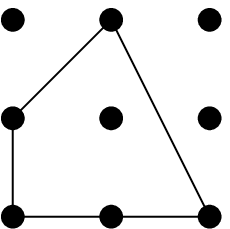}}
  \qquad
  \subfigure[\label{quad1-6}]{\includegraphics[height=1.5cm]{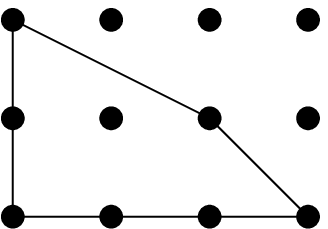}}
  \qquad
  \subfigure[\label{quad1-7}]{\includegraphics[height=1.5cm]{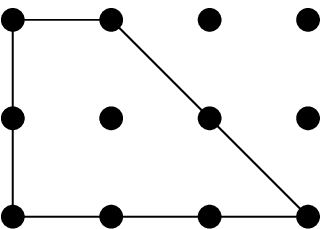}} \\
  %\qquad
  \subfigure[\label{pent1-1}]{\includegraphics[height=1.5cm]{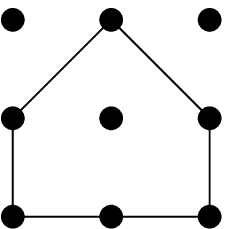}}
  \qquad
  \subfigure[\label{pent1-2}]{\includegraphics[height=1.5cm]{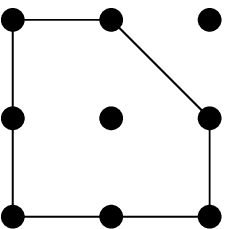}}
  \qquad
  \subfigure[\label{pent1-3}]{\includegraphics[height=1.5cm]{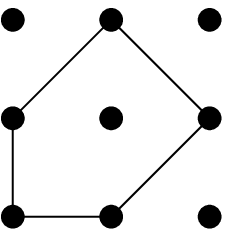}}
  \qquad
  \subfigure[\label{hex1-1}]{\includegraphics[height=1.5cm]{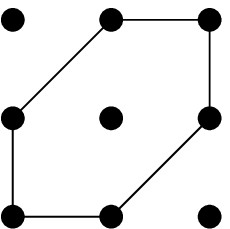}}
  \caption{All integral polygons with one interior integer point}
  \label{i=1}
\end{figure}
\end{theorem}

The only result from the previous sections that is used for the description of $\M$ is Lemma~\ref{properties of tight} dealing with $\cR^d(a)$. We use the description of $\cR^2(a)$ presented in the following remark. 

\begin{remark} \label{R^2(0)}
With the help of Theorem \ref{i=1:enum}, the set $\cR^2(a)$
can be computed for a given $a \in \Z^2$.
Let us assume $a = o$, since, by a unimodular
transformation, the choice of $a$ is not important.
Then, up to a unimodular transformation, every element of
$\cR^2(o)$ coincides with one of the following sets:
\begin{eqnarray*}
  R_1 & := & \conv(\{o,2 e_1\}), \\
  R_2 & := & \conv(\{o, 3 e_1, 2 e_2\}), \\
  R_3 & := & \conv(\{o, 2 e_1, e_1 + 2 e_2\}), \\
  R_4 & := & \conv(\{o,2 e_1 + e_2, 2 e_2 + e_1\}).
\end{eqnarray*}
This can be seen as follows.
By Lemma~\ref{properties of tight} \ref{a-red is simplex} 
and \ref{a-red only one int}, all elements of $\cR^2(o)$ are
simplices with precisely one relative interior integer
point.
Thus, up to a unimodular transformation, all two-dimensional
elements of $\cR^2(o)$ appear in
Figures~\ref{tria1-1}--\ref{tria1-5}.
Using Lemma~\ref{properties of tight}
\ref{a-red a is vertex} and \ref{a-red opp is empty}, we
end up with $R_2$, $R_3$, and $R_4$.
The fact that $R_1$ is the only one-dimensional element of
$\cR^2(o)$ is straightforward to verify.
\end{remark}

%%%%%%%%%%%%%%%%%%%%%%%%%%%%%%%%%%%%%%%%%%%%%%%%%%%%%%%%%%%%%%%%%%%%%%%
%% 6 Facets %%
%%%%%%%%%%%%%%%%%%%%%%%%%%%%%%%%%%%%%%%%%%%%%%%%%%%%%%%%%%%%%%%%%%%%%%%

\section{Elements in $\boldsymbol{\M}$ with six facets} \label{parity}

In this section we show that there exists, up to unimodular
transformation, only one $P \in \M$ with six facets.

\begin{lemma} \label{6.facet.lemma}
Let $P \in \M$.
Then, $P$ has at most six facets.
Furthermore, if $P$ has six facets, then each facet of $P$
is either the parallelogram shown in Figure~\ref{quad1-2} or
the triangle shown in Figure~\ref{tria1-3}.
\end{lemma}

\begin{proof}
We first show that $P$ has at most six facets.
Let $\F$ be the set of all facets of $P$.
We choose two integer points $p_1,p_2$ on an edge of $P$
with $[p_1,p_2] \cap \Z^3 = \{p_1,p_2\}$.
For each $F \in \F$ we fix an integer point $p_F$ in the
relative interior of $F$ in the following way.
If $F \in \F$ and $p_1,p_2 \in F$ let $p_F$ be a point in
$\relintt(F) \cap \Z^3$ such that the triangle with vertices
$p_1,p_2,p_F$ has minimal area.
This ensures $[p_F,p_i] \cap \Z^3 = \{p_F,p_i\}$ for $i =
1,2$.
If $F \in \F$ and $F \cap \{p_1,p_2\} = \{p_i\}$ for some
$i=1,2$, let $p_F$ be a point in $\relintt(F) \cap \Z^3$
with $[p_F,p_i] \cap \Z^3 = \{p_F,p_i\}$.
If $F \in \F$ and $F \cap \{p_1,p_2\} = \emptyset$, we
choose $p_F$ to be any point in $\relintt(F) \cap \Z^3$.
Let $X := \{p_1,p_2\} \cup \{p_F: F \in \F\}$.
By construction, all points in $X$ have different parity.
Hence, $|\F| = |X| - 2 \leq |\integer^d / 2 \integer^d| - 2 = 2^3 - 2 =6$. 

Let us now show the second part of the assertion.
For that, we first show that each facet of $P$ contains
exactly one integer point in its relative interior.
Assume there exists a facet $F_1$ containing at least two
integer points in its relative interior.
Choose a vertex $v_1$ of $F_1$ and two integer points
$p_1,p_2 \in \relintt(F_1) \cap \Z^3$ such that the triangle
with vertices $v_1,p_1,p_2$ has minimal area.
Let $e = [v_1,v_2]$ be an edge of $P$ which is not contained
in $F_1$ and let $\bar{v}_2$ be the integer point on the
edge $e$ which is closest to $v_1$.
Let $F_2$ and $F_3$ be the two facets containing both $v_1$
and $\bar{v}_2$.
Let $p_3$ (resp.~$p_4$) be an integer point in the relative
interior of $F_2$ (resp.~$F_3$) such that $[v_1,p_i] \cap
\Z^3 = \{v_1,p_i\}$ and $[\bar{v}_2,p_i] \cap \Z^3 =
\{\bar{v}_2,p_i\}$ for $i = 3,4$ (this can again be achieved
by choosing triangles with minimal area).
In the remaining three facets choose arbitrary relative
interior integer points $p_5, p_6, p_7$ such that
$[\bar{v}_2,p_i] \cap \Z^3 = \{\bar{v}_2,p_i\}$ for $i =
5,6,7$.
By construction, the points $v_1, \bar{v}_2, p_1, \dots,
p_7$ must have different parity which is a contradiction.

Let now $F$ be an arbitrary facet of $P$.
It follows that $F$ is one of the polygons shown in
Figure~\ref{i=1}.
If $F$ is different from the quadrilateral \ref{quad1-2} and
the triangle \ref{tria1-3}, then it contains four integer
points with different parity.
These four integer points together with the five interior
integer points of the other five facets of $P$ are nine
points of different parity which is a contradiction.
\end{proof}

The next lemma shows that all facets of a polytope $P \in
\M$ with six facets are quadrilaterals as pictured in
Figure~\ref{quad1-2} and thus, the shape of $P$ is uniquely
determined.

\begin{lemma}
Each $P \in \M$ with six facets is a parallelepiped where
each of the six facets is a parallelogram as depicted in
Figure~\ref{quad1-2}.
\end{lemma}

\begin{proof}
By Lemma \ref{6.facet.lemma}, $P$ has only two types of
facets.
Since quadrangular facets do not contain edges with relative
interior integer points, it follows that $P$ has an even
number of triangular facets and that these facets are
pairwise attached.
In \cite[Sections 6{.}2 and 6{.}3]{Gruenbaum03} all possible combinatorial types of three-dimensional polytopes with six facets are enumerated (there are exactly seven such types).
Since all facets of $P$ are quadrilaterals shown in
Figure~\ref{quad1-2} or triangles shown in
Figure~\ref{tria1-3}, and since triangular facets occur
pairwise, we deduce that $P$ is one of the three combinatorial
types in Figure~\ref{comb.types}. 

\begin{figure}[ht]
  \centering
  \mbox{
    \subfigure[Type A]{\includegraphics[height=2cm]{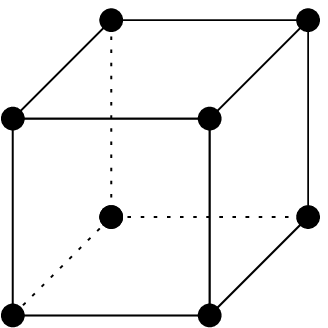}}
    \qquad
    \subfigure[Type B]{\includegraphics[height=2.6cm]{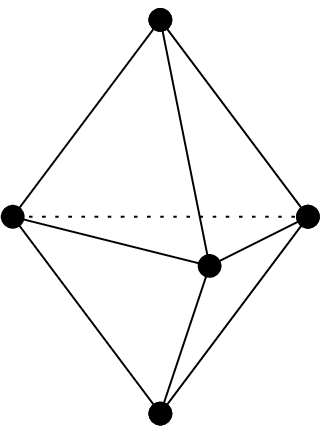}}
    \qquad
    \subfigure[Type C]{\includegraphics[height=2.6cm]{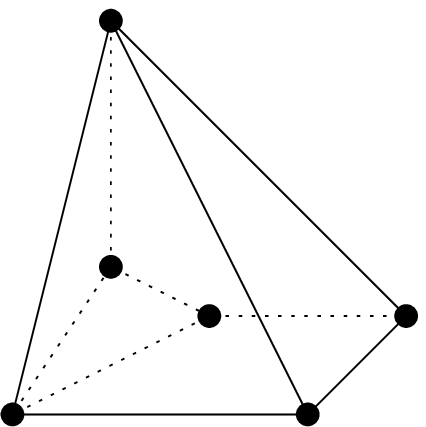}}
  }
  \caption{Possible combinatorial types of $P$}
  \label{comb.types}
\end{figure}

First assume that $P$ is of combinatorial type B, having
only triangular facets.
Since all facets contain exactly one edge with exactly one
relative interior integer point, only two different
polytopes $P$ are possible as depicted in
Figure~\ref{tetra}, where the gray nodes represent integer
points on edges.
In both cases, the three integer points represented in gray
together with the six relative interior integer points of
the six facets of $P$ are nine points of different parity
which is a contradiction.
Thus, $P$ cannot be of combinatorial type B.

\begin{figure}[ht]
  \centering
  \mbox{
    \subfigure[\label{tetra} Polytope of type B]{\includegraphics[height=2.6cm]{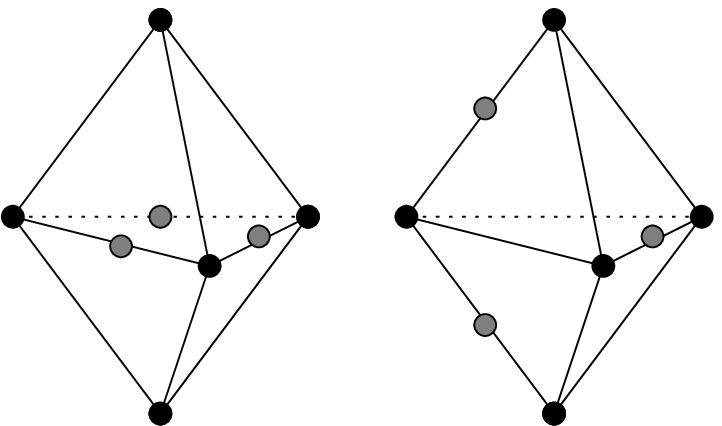}}
    \qquad
    \subfigure[\label{pyra} Polytope of type C]{\includegraphics[height=2.6cm]{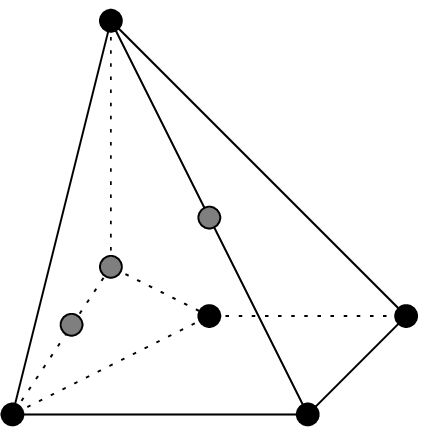}}
  }
  \caption{Polytopes $P$ of combinatorial types B and C}
  \label{impossible}
\end{figure}

Now assume that $P$ is of combinatorial type C, having two
quadrangular and four triangular facets.
Then, the location of the two relative interior integer
points on its edges is already determined by the structure
of the facets of $P$ as illustrated in Figure~\ref{pyra}.
These two points together with a particular vertex of $P$
(the gray nodes in Figure~\ref{pyra}) and the six relative
interior integer points of the six facets of $P$ are nine
points of different parity.
Thus, $P$ cannot be of combinatorial type C.

It follows that $P$ must be of combinatorial type A.
This implies that all facets of $P$ are quadrangular and
therefore $P$ has the shape depicted in Figure~\ref{M12}.
\end{proof}

%%%%%%%%%%%%%%%%%%%%%%%%%%%%%%%%%%%%%%%%%%%%%%%%%%%%%%%%%%%%%%%%%%%%%%%
%% 5 Facets %%
%%%%%%%%%%%%%%%%%%%%%%%%%%%%%%%%%%%%%%%%%%%%%%%%%%%%%%%%%%%%%%%%%%%%%%%

\section{Elements in $\boldsymbol{\M}$ with five facets} \label{slicing}

By \cite[Section 6{.}1]{Gruenbaum03}, there are exactly two
combinatorial types of three-dimensional polytopes with five
facets. These are  quadrangular pyramids (i.e., pyramids having a
quadrangular base) and 
triangular prisms (i.e., prisms having triangular
bases). We will analyze both combinatorial types separately.

%%%%%%%%%%%%%%%%%%%%%%%%%%%%%%%%%%%%%%%%%%%%%%%%%%%%%%%%%%%%%%%%%%%%%%%
%% Pyramids %%
%%%%%%%%%%%%%%%%%%%%%%%%%%%%%%%%%%%%%%%%%%%%%%%%%%%%%%%%%%%%%%%%%%%%%%%

\subsection{Quadrangular pyramids} \label{quad-pyramids}

Let $P \in \M$ be a quadrangular pyramid.
Using a unimodular transformation, base $F$ and apex
$a = (a_1, a_2, a_3)$ of $P$ can be assumed to satisfy $F
\subseteq \R^2 \times \{0\}$ and $a_3 > 0$.
We can further assume that $a_3 \geq 2$ since for $a_3 = 1$,
$P$ is contained in $\real^2 \times [0,1]$ which is a contradiction to its maximality. 

We first show that there is only one quadrangular pyramid
$P \in \M$ with $a_3 = 2$ and $a_3 = 3$, respectively, up to
a unimodular transformation.

\begin{lemma}
Let $P \in \M$ be a quadrangular pyramid with base
$F \subseteq \R^2 \times \{0\}$ and apex $a = (a_1, a_2,
a_3)$, where $a_3 = 2$.
Then $P$ is equivalent to the pyramid $M_8$.
\end{lemma}

\begin{proof}
Let $F':= P \cap (\R^2 \times \{1\})$.
Since each triangular facet of $P$ contains an integer point
in its relative interior, it follows that $F'$ is a maximal
lattice-free quadrilateral and contains precisely four
integer points, one in the relative interior of each of the
edges of $F'$.
Without loss of generality assume $F'\cap \Z^3 = \{0,1\}^2
\times \{1\}$.
By convexity, $\vertt(F')$ lies in the union of $(0,1)
\times \R \times \{1\}$ and $\R \times (0,1) \times \{1\}$.
On the other hand $\vertt(F') = \frac{1}{2} a + \frac{1}{2}
\vertt(F) \subseteq \frac{1}{2} \Z^3$.
Hence $\vertt(F')$ lies in the union of $\{\frac{1}{2}\}
\times \frac{1}{2} \Z \times \{1\}$ and $\frac{1}{2} \Z
\times \{\frac{1}{2}\} \times \{1\}$.
Clearly, $\vertt(F')$ is disjoint with $[0,1]^2 \times
\{1\}$.
It follows that $F'$ contains the set $B := \frac{1}{2} e_1
+ \frac{1}{2} e_2 + e_3 + \conv(\{\pm e_1, \pm e_2\})$.
If $B$ were a proper subset of $F'$, then one of the points
from the set $\{0,1\}^2 \times \{1\}$ would be in the
relative interior of $F',$ a contradiction.
Hence $F'=B$.
We have determined that, up to a unimodular transformation,
$F$ is a translate of $\conv(\{\pm 2 e_1, \pm 2 e_2\})$ and
$F'$ is a translate of $B$ by an integer vector.
This implies the assertion.
\end{proof}

\begin{lemma}
Let $P \in \M$ be a quadrangular pyramid with base $F
\subseteq \R^2 \times \{0\}$ and apex $a = (a_1, a_2, a_3)$,
where $a_3 = 3$.
Then $P$ is equivalent to the pyramid $M_9$.
\end{lemma}

\begin{proof}
If $p \in P \cap (\R^2 \times \{2\})$ is an integer point in
the relative interior of a facet of $P$, then $2p - a \in P
\cap (\R^2 \times \{1\})$ is also an integer point in the
relative interior of the same facet of $P$.
Consequently, $F' := P \cap (\R^2 \times \{1\})$ contains
precisely four integer points, one in the relative interior
of each of its edges.
Without loss of generality assume $F'\cap \Z^3 = \{0,1\}^2
\times \{1\}$.
By convexity, $\vertt(F')$ lies in the union of $(0,1)
\times \R \times \{1\}$ and $\R \times (0,1) \times \{1\}$.
On the other hand $\vertt(F') = \frac{1}{3} a + \frac{2}{3}
\vertt(F) \subseteq \frac{1}{3} \Z^3$.
Hence $\vertt(F')$ lies in the union of
$\{\frac{1}{3},\frac{2}{3}\} \times \frac{1}{3} \Z \times
\{1\}$ and $\frac{1}{3} \Z \times
\{\frac{1}{3},\frac{2}{3}\} \times \{1\}$.
Clearly, $\vertt(F')$ is disjoint with $[0,1]^2 \times
\{1\}$.
A simple analysis of all possible cases reveals that, by a
unimodular transformation, only one $F'$ is possible and we
can assume that $F' := \frac{1}{3} e_1 + \frac{1}{3} e_2 +
e_3 + \conv(\{\frac{4}{3} e_1, -\frac{2}{3} e_1, \frac{4}{3}
e_2, -\frac{2}{3} e_2\})$.
Thus, up to a unimodular transformation, $F$ is a translate
of $\conv(\{2 e_1, -e_1, 2 e_2, -e_2\})$.
This implies the assertion. 
\end{proof}

In the following we assume that $a_3 \geq 4$ and show that
no further maximal lattice-free quadrangular pyramid $P$
exists.
The proof consists of the following steps.
First, we construct all bases which are possible for such
a pyramid $P$.
Second, we argue that only two of them can appear as bases
for $a_3 \geq 11$ and analyze these two separately.
Third, the other bases are ruled out by a computer
enumeration. 

We start with a lemma which shall be used later for
simplices in Section \ref{four.facets} as well.

\begin{lemma} \label{2ineq}
Let $P \in \M$ be a simplex or a quadrangular pyramid with
base $F \subseteq \R^2 \times \{0\}$ and apex $a = (a_1,
a_2, a_3)$, where $h := a_3 \geq 4$.
Then $w(F) = 2$ and the following inequalities hold:
\begin{equation} \label{bases.bound}
  2 i(F) + b(F) \le \floor{\frac{6h-4}{h-2}} \le 10.
\end{equation}
In particular, if $P$ is a simplex (resp.~a quadrangular
pyramid), then $(i(F),b(F)) \in Z_S$ (resp.~$Z_Q$),  where
\begin{align*}
Z_S :=& \{(1,j) : j = 3, \dots, 8\} \cup \{(2,j) : j = 3,
\dots, 6\}, \\ 
Z_Q :=& \{(1,j) : j = 4, \dots, 8\} \cup \{(2,j) : j = 4,
\dots, 6\}.
\end{align*}
\end{lemma}

\begin{proof} 
Let $F' := P \cap (\R^2 \times \{1\})$.
Since $F$ contains an integer point in its relative interior
we have $w(F) \geq 2$.
Assume that $w(F) \geq 3$.
Then $h \geq 4$ implies $w := w(F') = w(F) \frac{h-1}{h}
\geq \frac{9}{4} > 1 +\frac{2}{\sqrt{3}}$.
Hence, by Theorem \ref{area,width}, $F'$ is not lattice-free
which is a contradiction.
Thus, we have $w(F) = 2$ and it follows
$2 > w = w(F) \frac{h-1}{h} \geq \frac{3}{2}$.
Applying Theorem \ref{area,width} to $F'$, we obtain
\begin{equation} \label{quad.pyr}
  \area(F) = \left(\frac{h}{h-1}\right)^2 \area(F') \le
  \left(\frac{h}{h-1}\right)^2 \frac{w^2}{2(w-1)}
  = \frac{2h}{h-2},
\end{equation}
where the last equality follows from $w = 2 \frac{h-1}{h}$.
Consequently, combining \eqref{quad.pyr} and Pick's formula
and using the fact that $\floor{\frac{6h-4}{h-2}}$ is
monotonically nonincreasing for $h \geq 4$, we arrive at the
stated inequalities.

We now show that $i(F) \le 2$.
Assume the contrary, i.e., $i(F) \ge 3$.
Performing an appropriate unimodular transformation to $P$
we can assume that $\pi(F) = [o,2 e_1]$.
For $x \in \pi(P)$ let $f(x)$ be the length of the line
segment $\pi^{-1}(x) \cap P$.

The conditions $w(F)=2$ and $i(F) \ge 3$ imply $f(e_1) \ge 3$. By Lemma~\ref{projection non-free}, $\pi(P)$ contains an
integer point in its interior. The relative interior of $[e_1,\pi(a)]$ does not contain integer points, since otherwise the value of $f$ at the integer point in $[e_1,\pi(a)] \setminus \{e_1\}$ closest to $e_1$ would be $>1$ yielding a contradiction to the lattice-freeness of $P$. Thus, the interior of $\conv (\{o,e_1,\pi(a)\})$ or $\conv (\{e_1, 2 e_1, \pi(a)\})$ contains an integer point. By symmetry reasons, we may assume that for $T:=\conv (\{o,e_1, \pi(a)\})$ one has  $\intr (T) \cap \integer^2 \ne \emptyset$.

Let $R$ be an element of $\cR^2(e_1)$ contained in $T$ and such that the relative
interior of $R$ contains an interior integer point of
$T$.
Note that $R$ is equivalent to one of the polygons $R_1,
\ldots, R_4$ in Remark~\ref{R^2(0)}.

\emph{Case~1:} $R \equiv R_1 \modulo{\Aff(\integer^2)}$.
Then $R=[e_1,p]$ for some $p \in T \cap \Z^2$ and such
that the point $\frac{1}{2}(e_1+p)$ is integer and in the
interior of $T$.
By the concavity of $f$, one has 
\begin{align*}
f \left( \frac{1}{2}(e_1+p) \right) \ge \frac{1}{2} f(e_1) +
\frac{1}{2} f(p) \ge \frac{1}{2} f(e_1) \ge \frac{3}{2}  > 1.
\end{align*}
Thus, a contradiction to the lattice-freeness of $P$.

\emph{Case~2:} $R \equiv R_4 \modulo{\Aff(\integer^2)}$.
Then $R= \conv \{e_1, p, q\}$ for some $p, q \in T \cap
\integer^2$ and $\frac{1}{3}(e_1 + p + q)$ is integer and
in the interior of $T$.
By the concavity of $f$, we have 
\begin{align*}
f \left( \frac{1}{3} (e_1 + p + q) \right) \ge \frac{1}{3}
\Big( f(e_1) + f(p) + f(q) \Big) \ge \frac{1}{3} f(e_1) \ge
1.
\end{align*}
It follows that $f(p)=f(q)=0$, since otherwise one has $f \left( \frac{1}{3} (e_1 + p + q) \right) > 1$ yielding 
a contradiction to the lattice-freeness of $P$. Then, in view of the choice of $T$, we have $p, q \in [o,\pi(a)]$. The equality $\{p,q\} = \{o,\pi(a)\}$ would imply that $a_3=3$ contradicting the assumption. Thus, one of the points $p$, $q$ (say $p$) lies in the relative interior of $[o,\pi(a)]$. We use the point $2p-q$, which is the integer point on $[o,\pi(a)] \setminus [p,q]$ closest to $p$. 

We shall use the following property of $R_4$.
Let $r_1, r_2, r_3$ be the vertices of $R_4$.
Then the segment joining $r_1$ and $2 r_2 - r_3$ (the
reflection of $r_3$ with respect to $r_2$) contains
precisely two integer points in its relative interior.
Consider the subcase that the point $2p -q$ lies in the
relative interior of $[o,\pi(a)]$.
Then the relative interior of $[e_1,2p-q]$ is contained in
the interior of $T$.
Taking into account the indicated property of $R_4$ we see
that the relative interior of $[e_1,2p-q]$ contains two
integer points.
Thus, applying the arguments as in Case~1, we arrive at a
contradiction.
For the subcase that the point $2p -q$ coincides with $o$ or
$\pi(a)$, the fact that the relative interior of
$[e_1,2p-q]$ contains two integer points contradicts the
fact that the segments $[o,e_1]$ and $[e_1,\pi(a)]$ do not
contain integer points in their relative interiors.

\emph{Case~3:} $R \equiv R_i \modulo{\Aff(\integer^2)}$ for
$i \in \{2,3\}$.
Then there exists an edge $e$ of $R$ incident to $e_1$ which
contains at least three integer points.
Since the edge $[o,2e_1]$ of $\pi(P)$ contains three integer
points and the integer point $e_1$ is between the two
remaining integer points, it follows that the edge $e$ is
not contained in the boundary of $\pi(P)$.
Thus, on $e$ we can find an integer point $p$ such that
$\frac{1}{2}(e_1 + p)$ is integer and in the interior of
$\pi(P)$.
But then, applying the same arguments as in Case~1 we
arrive at a contradiction.

So far, we have shown that $i(F) \in \{1,2\}$ and $2 i(F) +
b(F) \le 10$. 
If $P$ is a simplex, then $b(F) \geq 3$.
Thus, $(i(F),b(F)) \in Z_S$ in this case.
If $P$ is a quadrangular pyramid, then $b(F) \geq 4$.
Thus, we have $(i(F),b(F)) \in Z_Q$.
\end{proof}

In order to analyze quadrangular pyramids $P \in \M$ further
we need a list of all integral quadrilaterals $Q$ in the
plane with $w(Q) = 2$ and $(i(Q),b(Q)) \in Z_Q$ since these
quadrilaterals are candidates for the base of $P$.
By \eqref{bases.bound}, it follows that $2 i(F) + b(F) \leq
6$ for $a_3 \geq 11$ which implies that the base $F$ of such
a pyramid has exactly one integer point in its relative
interior and exactly the four vertices as the only integer
points on its boundary.
From Figure~\ref{i=1}, it follows that only two
quadrilaterals qualify as a base for $P$ in this case
(Figure~\ref{quad1-1} and \ref{quad1-2}).
We will analyze these two possible bases separately from the
others. However, we will first prove the following lemma.

\begin{lemma} \label{i=2,i=3}
Let $Q \subseteq \R^2$ be an integral quadrilateral with
$w(Q) = 2$, $i(Q) = 2$, and $b(Q) \in \{4,5,6\}$.
Then, up to a unimodular transformation, $Q$ is one of the
quadrilaterals depicted in Figure~\ref{i=2}.
\begin{figure}[ht]
  \centering
  \subfigure[\label{quad2-1}]{\includegraphics[height=1.5cm]{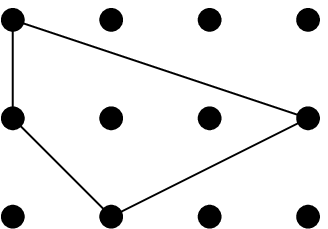}}
  \qquad
  \subfigure[\label{quad2-2}]{\includegraphics[height=1.5cm]{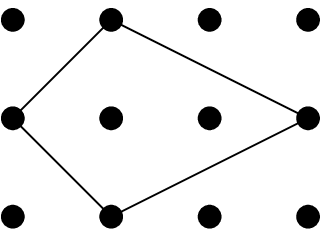}}
  \qquad
  \subfigure[\label{quad2-3}]{\includegraphics[height=1.5cm]{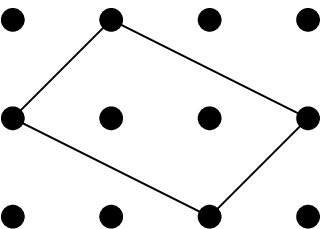}}
  \qquad
  \subfigure[\label{quad2-4}]{\includegraphics[height=1.5cm]{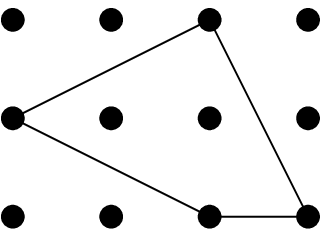}}
  \qquad
  \subfigure[\label{quad2-5}]{\includegraphics[height=1.5cm]{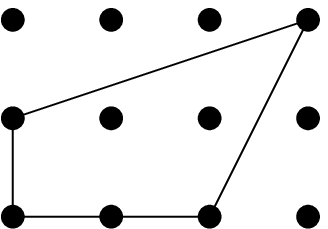}} \\
  \subfigure[\label{quad2-6}]{\includegraphics[height=1.5cm]{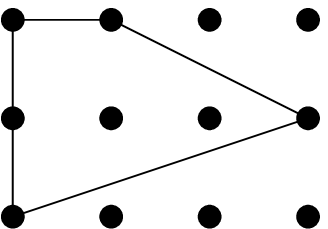}}
  \qquad
  \subfigure[\label{quad2-7}]{\includegraphics[height=1.5cm]{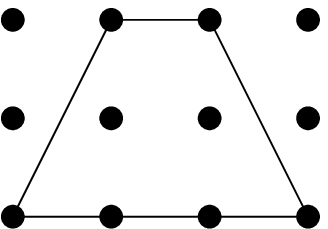}}
  \qquad
  \subfigure[\label{quad2-8}]{\includegraphics[height=1.5cm]{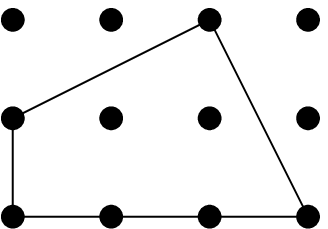}}
  \qquad
  \subfigure[\label{quad2-9}]{\includegraphics[height=1.5cm]{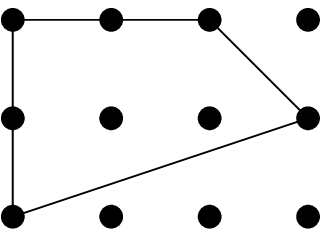}}
  \qquad
  \subfigure[\label{quad2-10}]{\includegraphics[height=1.5cm]{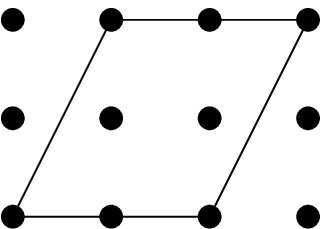}}
  \caption{All integral quadrilaterals $Q$ with $w(Q) = 2$,
    $i(Q) = 2$, and $b(Q) \in \{4,5,6\}$}
  \label{i=2}
\end{figure}

\end{lemma}

\begin{proof}
Let $Q$ be an integral quadrilateral in the plane satisfying
$w(Q) = 2$ and $i(Q) = 2$.
We divide the proof according to the number of integer
points on the boundary of $Q$.

\emph{Case~1:}  $i(Q) = 2$ and $b(Q) = 4$.
Pick's formula gives $\area(Q) = 3$ in this case.
Without loss of generality we assume that the two interior
integer points are placed at $(1,0)$ and $(2,0)$.
This implies that for any $u \in \Z^2 \setminus \{o, \pm
e_2\}$ we have $w(Q,u) \geq 3$ and therefore it must hold
$v_2 \in \{0, \pm 1\}$ for each vertex $v = (v_1, v_2)$ of
$Q$.
We distinguish three subcases based on the number of
vertices of $Q$ that lie on the line $y = 0$.

\emph{Subcase~1a:} Two vertices of $Q = \conv(\{a,b,c,d\})$ lie on the
line $y = 0$.
Then, one vertex is $a = (0,0)$ and the other $c = (3,0)$.
Let the remaining two vertices $b$ and $d$ satisfy $d_2 = 1
= - b_2$.
We can assume that $d = (0,1)$ for if $d = (d_1,1)$ we apply
the unimodular transformation
$(x,y) \mapsto (x - d_1 y,y)$.
For convexity reasons it follows that $b \in \{(1,-1),
(2,-1), (3,-1), (4,-1), (5,-1)\}$.
Choices $b = (1,-1)$ and $b = (5,-1)$ are equivalent and
lead to the quadrilateral shown in Figure~\ref{quad2-1},
$b = (2,-1)$ and $b = (4,-1)$ lead to Figure~\ref{quad2-2}
and $b = (3,-1)$ leads to Figure~\ref{quad2-3}.

\emph{Subcase 1b:} One vertex of $Q = \conv(\{a,b,c,d\})$ lies on the
line $y = 0$. 
Without loss of generality assume that $a = (0,0)$ and
$b$, $c$, and $d$ satisfy $b_2 = 1 = - c_2 = - d_2$.
It follows that $c_1 = d_1 + 1$ since $b(Q) = 4$, by
assumption.
Without loss of generality we can place $b$ at $(0,1)$.
By convexity of $Q$ and since $(1,0)$ and $(2,0)$ are the
only interior integer points of $Q$ we obtain $c = (5,-1)$
and $d = (4,-1)$ giving the quadrilateral shown in
Figure~\ref{quad2-4}.

\emph{Subcase 1c:} No vertex of $Q = \conv(\{a,b,c,d\})$ lies on the line
$y = 0$.
Without loss of generality let $a_2 = b_2 = 1 = - c_2 =
- d_2$. 
It follows that $b_1 = a_1 + 1$ and $c_1 = d_1 + 1$.
Thus, $\area(Q) = 2$ which contradicts Pick's formula.

\emph{Case 2:} $i(Q) = 2$ and $b(Q) = 5$.
Pick's formula gives $\area(Q) = 3{.}5$.
Placing the two interior integer points of $Q$ at $(1,0)$
and $(2,0)$ as above implies again that $v_2 \in \{0, \pm
1\}$ for each vertex $v = (v_1, v_2)$ of $Q$.
If two vertices of $Q$ lie on the line $y = 0$, then $Q$ has
no edge with a relative interior integer point, a
contradiction to $b(Q) = 5$.
If no vertex of $Q$ lies on the line $y = 0$, then $\area(Q)
= 3$, a contradiction to Pick's formula.
Thus, precisely one vertex of $Q = \conv(\{a,b,c,d\})$ lies
on the line $y = 0$.
Place it at $a = (0,0)$.
Without loss of generality let $b_2 = c_2 = -1 = - d_2$.
Using an appropriate unimodular transformation we can assume
that $d = (0,1)$.
Thus, either the edge connecting $b$ and $c$ or the edge
connecting $c$ and $d$ has a relative interior integer point
which is $\frac{1}{2}(b+c)$ or $\frac{1}{2}(c+d)$,
respectively.
In the first case we end up with $b = (3,-1)$ and $c =
(5,-1)$ (Figure~\ref{quad2-5}), whereas the latter leads to
$b = (5,-1)$ and $c = (6,-1)$ (Figure~\ref{quad2-6}).

\emph{Case 3:}  $i(Q) = 2$ and $b(Q) = 6$. 
Pick's formula gives $\area(Q) = 4$.
Placing the two interior integer points of $Q$ at $(1,0)$
and $(2,0)$ as above implies again that $v_2 \in \{0, \pm
1\}$ for each vertex  $v = (v_1, v_2)$ of $Q$.
If two vertices of $Q$ lie on the line $y = 0$, then $Q$ has
no edge with a relative interior integer point, a
contradiction to $b(Q) = 6$. 
We consider two subcases.

\emph{Subcase 3a:} No vertex of $Q = \conv(\{a,b,c,d\})$ lies on the line
$y = 0$.
Without loss of generality assume that $a_2 = b_2 = 1 =
- c_2 = - d_2$.
We either have $b_1 = a_1 + 2$ and $c_1 = d_1 + 2$ or $b_1 =
a_1 + 1$ and $c_1 = d_1 + 3$.
Using an appropriate unimodular transformation we can assume
that $a = (0,1)$.
Then, the first case leads to $b = (2,1)$, $c = (3,-1)$, and
$d = (1,-1)$ (Figure~\ref{quad2-10}), whereas the latter
leads to $b = (1,1)$, $c = (4,-1)$, and $d = (1,-1)$
(Figure~\ref{quad2-7}).

\emph{Subcase 3b:} One vertex of $Q = \conv(\{a,b,c,d\})$ lies on the
line $y = 0$.
Without loss of generality assume that $a = (0,0)$ and
$b$, $c$, and $d$ satisfy $b_2 = 1 = - c_2 = - d_2$.
Using an appropriate unimodular transformation we can assume
that $b = (0,1)$. 
Then, the edge connecting $c$ and $d$ has either two or one
relative interior integer points.
In the first case we obtain $c = (5,-1)$ and $d = (2,-1)$
(Figure~\ref{quad2-8}).
In the second case both edges, the one connecting $c$ and
$d$ and the one connecting $b$ and $c$ have each one
relative interior integer point and it follows $c = (6,-1)$
and $d = (4,-1)$ (Figure~\ref{quad2-9}).
\end{proof}

Lemma \ref{i=2,i=3} completes the list of the possible
bases of a quadrangular pyramid $P \in \M$: precisely the
quadrilaterals shown in Figure~\ref{quad1-1}--\ref{quad1-7}
and \ref{i=2} qualify for a base of $P$.
We will now show that there is no quadrangular pyramid
$P \in \M$ with $a_3 \geq 11$.

\begin{lemma} \label{DiamondBasis}
Let $P \subseteq \R^3$ be a pyramid with base $\conv(\{ \pm
e_1, \pm e_2 \})$ and apex $a = (a_1, a_2, a_3) \in \Z^3$,
where $a_3 \geq 4$.
Then $P$ is not maximal lattice-free.
\end{lemma}

\begin{proof}
By an appropriate unimodular transformation we can assume
that $0 \le a_i < a_3$ for $i = 1,2$.
We represent the base by $F := \conv(\{ \pm e_1, \pm e_2 \})
= \{y \in \R^3: |y_1|+|y_2| \le 1,\ y_3 = 0\}$.
Then, $P = \{x \in \R^3: x = (1 - \lambda) y + \lambda a
\mbox{ for } 0 \le \lambda \le 1 \mbox{ and } y \in F\}$ and
therefore
\begin{align*}
  \intt(P) & = \{x \in \R^3: x = (1 - \lambda) y + \lambda a
               \mbox{ for } 0 < \lambda < 1 \mbox{ and } y
               \in \relintt(F)\} \\
           & = \{x \in \R^3: \frac{1}{1-\lambda} x -
               \frac{\lambda}{1-\lambda} a \in \relintt (F)
               \text{ for some } 0 < \lambda < 1\} \\
           & = \{x \in \R^3: |x_1 - \lambda a_1| + |x_2 -
               \lambda a_2| < 1 - \lambda \text{ and } x_3 =
               \lambda a_3 \text{ for some } 0 < \lambda <
               1\}.
\end{align*}
It follows 
\begin{equation} \label{ZCapIntrP}
  \Z^3 \cap \intt(P) = \{x \in \Z^3: |a_3 x_1 - a_1 x_3| +
  |a_3 x_2 - a_2 x_3| < a_3 - x_3,\  x_3 \in \{1, \dots, a_3
  - 1\} \}.
\end{equation}
From \eqref{ZCapIntrP} we derive the following equivalences:
\begin{itemize} 
  \item $(0,0,1) \in \intt(P)$ if and only if $a_1 + a_2 < a_3 - 1$;
  \item $(1,1,1) \in \intt(P)$ if and only if $a_1 + a_2 > a_3 + 1$; 
  \item $(1,0,1) \in \intt(P)$ if and only if $a_1 - a_2 > 1$;
  \item $(0,1,1) \in \intt(P)$ if and only if $a_2 - a_1 > 1.$ 
\end{itemize} 
If one of the above mentioned conditions is fulfilled, $P$
is not lattice-free.
We can therefore assume that the following two inequalities
are satisfied:
\begin{eqnarray}
  |a_1 + a_2 - a_3| & \le & 1, \label{RemCase1} \\
  |a_1 - a_2|       & \le & 1. \label{RemCase2}
\end{eqnarray}
It can be verified directly that for $a_1,a_2,a_3$
satisfying \eqref{RemCase1} and \eqref{RemCase2} one has
$|a_3 - 2a_1| + |a_3 - 2a_2| \le 2$.
In view of \eqref{ZCapIntrP}, $(1,1,2) \in \intt(P)$ if and
only if $|a_3 - 2 a_1|+|a_3 - 2 a_2| < a_3-2$.
Hence, when \eqref{RemCase1} and \eqref{RemCase2} are
fulfilled, $(1,1,2)$ is an interior point of $P$ if $a_3 >
4$.
It remains to exclude the case $a_3 = 4$.
Integer vectors $a = (a_1,a_2,a_3)$ satisfying
\eqref{RemCase1}, \eqref{RemCase2} and $a_3=4$ are precisely
vectors from the set
$\{(2,2,4),(2,1,4),(3,2,4),(1,2,4),(2,3,4)\}$.
All these vectors do not correspond to maximal lattice-free
pyramids.
\end{proof}

\begin{lemma} \label{ArrowBasis}
Let $P \subseteq \R^3$ be a pyramid with base
$\conv(\{ e_1, e_2, \pm (e_1 + e_2) \})$ and apex
$a = (a_1, a_2, a_3) \in \Z^3$, where $a_3 \geq 4$.
Then $P$ is not maximal lattice-free.
\end{lemma}

\begin{proof}
By an appropriate unimodular transformation we can assume
that $0 \le a_i < a_3$ for $i = 1,2$.
The set $\conv(\{e_1, e_2, \pm(e_1+e_2)\})$ is the set of
all $y = (y_1, y_2, y_3) \in \R^3$ satisfying 
\begin{align*}
  y_1 & \le 1, & y_1 - 2 y_2 & \le 1, & y_3 & = 0, \\
  y_2 & \le 1, & y_2 - 2 y_1 & \le 1. & &
\end{align*}
By this, $\intt(P)$ is the set of all $x = (x_1, x_2, x_3)
\in \R^3$ satisfying
\begin{align*}
  x_1 - \lambda a_1 & < 1 - \lambda, &
  x_1 - \lambda a_1 - 2 (x_2 - \lambda a_2) &
  < 1 - \lambda, & x_3 & = \lambda a_3, \\
  x_2 - \lambda a_2 &
  < 1 - \lambda, & x_2 - \lambda a_2 - 2 (x_1 - \lambda a_1) &
  < 1 - \lambda & & 
\end{align*}
for some $0 < \lambda < 1$. Consequently, $\Z^3 \cap
\intt(P)$ is the set of all $x = (x_1, x_2, x_3) \in \Z^3$
satisfying
\begin{align*}
  a_3 x_1 + (1 - a_1) x_3 & < a_3, &
  a_3 x_1 - 2a_3 x_2 + (1 - a_1 + 2a_2) x_3 & < a_3, &
  x_3 & \in \{1,\ldots, a_3 - 1\}, \\
  a_3 x_2 + (1 - a_2) x_3 & < a_3, &
  a_3 x_2 - 2a_3 x_1 + (1 - a_2 + 2a_1) x_3 & < a_3.
\end{align*}
From these inequalities we obtain that $(1,1,1) \in
\intt(P)$ if and only if $a_1 > 1$ and $a_2 > 1$.
Hence, lattice-freeness requires that $a_1 \in \{0,1\}$ or
$a_2 \in \{0,1\}$.
By symmetry, it suffices to consider the cases $a_1 = 0$ and
$a_1 = 1$.

\emph{Case 1}: $a_1 = 0$.
If $a_2 > 1$, then $(0,1,1) \in \intt(P)$.
Otherwise $(0,0,1) \in \intt(P)$.

\emph{Case 2}: $a_1 = 1$.
If $a_2 > 3$, then $(0,1,1) \in \intt(P)$.
Thus, we have $a_2 \le 3$.
If $2 a_2 < a_3$, then $(0,0,1) \in \intt(P)$.
So we have $2 a_2 \ge a_3$ and it follows $a_3 \in
\{4,5,6\}$.
Hence, $a \in \{(1,2,4), (1,3,4), (1,3,5), (1,3,6)\}$.
All these vectors do not correspond to maximal lattice-free
pyramids.
\end{proof}

Lemmas \ref{DiamondBasis} and \ref{ArrowBasis} restrict
potential quadrangular pyramids $P \in \M$ to satisfy $4
\leq a_3 \leq 10$.
Since, in addition, the set of possible bases is known from
Figures~\ref{quad1-3}--\ref{quad1-7} and \ref{i=2} we are
left with a finite list of quadrangular candidate pyramids.
Computer enumeration shows that none of them is maximal
lattice-free.

%%%%%%%%%%%%%%%%%%%%%%%%%%%%%%%%%%%%%%%%%%%%%%%%%%%%%%%%%%%%%%%%%%%%%%%
%% Prisms %%
%%%%%%%%%%%%%%%%%%%%%%%%%%%%%%%%%%%%%%%%%%%%%%%%%%%%%%%%%%%%%%%%%%%%%%%

\subsection{Triangular prisms}

Let $P \in \M$ be a triangular prism.
We first show that the two triangular bases of $P$ are
translates.

\begin{lemma} \label{shape.prism}
Let $P \in \M$ be combinatorially equivalent to a triangular
prism.
Then $P$ is a prism, i.e., the two bases of $P$ are parallel
translates.
\end{lemma}

\begin{proof}
Let $H_1$, $H_2$, and $H_3$ be the hyperplanes containing
the quadrilateral facets of $P$.
We show that $H_1$, $H_2$, and $H_3$ do not share a point.
Assume the contrary and choose $p \in H_1 \cap H_2 \cap
H_3$.
Let $T_2$ be the triangular facet of $P$ such that the
pyramid $S$ with base $T_2$ and apex $p$ contains $P$.
Let $T_1$ be the triangular facet of $P$ distinct from
$T_2$.
Let $q$ be a vertex of $T_2$ closest to $\aff(T_1)$ and let
$H$ be the hyperplane parallel to $\aff(T_1)$ and passing
through $q$.
If $T_1$ and $T_2$ are not parallel, then the relative
interior of $P \cap H$ is contained in the interior of $P$.
On the other hand $T_1 + q - r$, where $r$ is the integer
point $r = T_1 \cap [p,q]$, is contained in $P \cap H$.
Hence the relative interior of $P \cap H$ contains an
integer point, a contradiction.
Thus, $T_1$ and $T_2$ are parallel.
Then, since $T_2$ is a base of $P$ and $T_1$ is a section of
$S$ parallel to $T_2$, we infer that $T_1$ and $T_2$ are
homothetic.
By construction, $T_1$ is strictly smaller than $T_2$.
Since $T_1$ is an integral triangle which contains at least
one integer point in its relative interior we have $w(T_1)
\geq 2$.
Therefore, since $T_2$ is integer and strictly larger,
$w(T_2) \geq 3$.
Without loss of generality we assume that $T_2 \subseteq
\R^2 \times \{0\}$ and $T_1 \subseteq \R^2 \times \{h\}$
with $h \ge 2$ ($h = 1$ do not need to be considered since
the quadrangular facets of $P$ contain integer points in
their relative interior).
Let now $T':= P \cap (\R^2 \times \{1\})$.
It follows that
\[ w(T') = \frac{h-1}{h} w(T_2) + \frac{1}{h} w(T_1) \ge
\frac{3(h-1)+2}{h} = 3 - \frac{1}{h} \ge \frac{5}{2} > 1 +
\frac{2}{\sqrt{3}},\]
a contradiction to \eqref{lwidth.bound} in Theorem
\ref{area,width}, since $T'$ is a lattice-free polygon in
$\R^2 \times \{1\}$ with respect to the lattice $\Z^2 \times
\{1\}$.
Hence $H_1$, $H_2$, and $H_3$ do not share a point and $P$
is a prism.
\end{proof}

According to Lemma \ref{shape.prism} it suffices to
investigate triangular prisms $P \in \M$ whose triangular
facets are parallel translates.
Without loss of generality we assume that the triangular
facets $T_1,T_2$ of $P$ satisfy $T_2 \subseteq \R^2 \times
\{0\}$ and $T_1 \subseteq \R^2 \times \{h\}$ with $h \ge
2$.
From Theorem \ref{area,width} and the fact that $P$ is
lattice-free, it follows that the hyperplane $H := \R^2
\times\{1\}$ satisfies $w(P \cap H) \leq 1 +
\frac{2}{\sqrt{3}}$.
Hence, $1 + \frac{2}{\sqrt{3}} \geq w(P \cap H) = w(T_2)
\geq 2$ and since $w(T_2) \in \Z$ we obtain $2 = w(T_2) =
w(P \cap H)$.
Theorem \ref{area,width} yields $2 \geq \area(P \cap H) =
\area(T_2)$ and Pick's formula gives $2i(T_2) + b(T_2) \leq
6$ implying $i(T_2) = 1$ and $b(T_2) \in \{3,4\}$.
Thus, by Figure~\ref{i=1}, $P$ has two triangular facets
which are either the triangle shown in Figure~\ref{tria1-5}
or the triangle shown in Figure~\ref{tria1-3}.
We prove that for each of these two cases there exists
exactly one maximal lattice-free triangular prism, up to
a unimodular transformation.

\begin{lemma}
Let $P \in \M$ be a triangular prism whose triangular facets
are the triangle shown in Figure~\ref{tria1-5}. 
Then, $P$ is equivalent to $M_{10}$.
\end{lemma}

\begin{proof}
Without loss of generality we assume that the two triangular
facets of $P$, denoted $F$ and $F'$, are given by
$F := \conv(\{e_1, e_2, -(e_1 + e_2)\})$ and $F' := a + F$,
where $a = (a_1, a_2, a_3)$ is the integer point in the
relative interior of $F'$.
By applying an appropriate unimodular transformation we can
further assume that $0 \leq a_i < a_3$ for $i = 1,2$.
Since the quadrangular facets of $P$ need to contain integer
points in their relative interior it holds $a_3 \geq 2$.
By symmetry, we assume $a_1 \leq a_2$.
In particular, we have $a_2 \geq 1$, otherwise $(0,0,1) \in
\intt(P)$.
We now set up the facet description of $P$ which is only
dependent on the parameters $a_1$, $a_2$, and $a_3$.
It follows that $\Z^3 \cap \intt(P)$ is the set of all
$x = (x_1, x_2, x_3) \in \Z^3$ satisfying
\begin{align*}
  a_3 x_1 - 2 a_3 x_2 + (2a_2 - a_1) x_3 & < a_3, &
  a_3 x_1 + a_3 x_2 - (a_1 &+ a_2) x_3 < a_3, \\
  a_3 x_2 - 2 a_3 x_1 + (2a_1 - a_2) x_3 & < a_3, &
  x_3 \in \{1,\ldots, &\,a_3 - 1\}.
\end{align*}
From these inequalities we obtain the following
equivalences:
\begin{itemize} 
  \item $(0,0,1) \in \intt(P)$ if and only if $-a_1 + 2a_2 < a_3$;
  \item $(0,1,1) \in \intt(P)$ if and only if $2a_1 < a_2$;  
  \item $(1,1,1) \in \intt(P)$ if and only if $a_3 < a_1 + a_2$. 
\end{itemize}
This implies that the following inequalities hold:
\begin{eqnarray}
  a_1 + a_3 & \le & 2 a_2, \label{kite.1} \\
  a_2       & \le & 2 a_1, \label{kite.2} \\
  a_1 + a_2 & \le & a_3. \label{kite.3}
\end{eqnarray}
Adding \eqref{kite.1} and \eqref{kite.3} yields $2 a_1 \leq
a_2$ and together with \eqref{kite.2} we obtain $a_2 = 2
a_1$.
Substituting this into \eqref{kite.1} and \eqref{kite.3}
leads to $a_3 \leq 3 a_1$ and $3 a_1 \leq a_3$ which means
that $a_3 = 3 a_1$.
It follows that $a = (a_1, 2a_1, 3a_1)$ for some $a_1 \geq
1$.
We infer that $(1,2,3) \in \intt(P)$ if $a_1 \geq 2$.
Thus, we have $a = (1,2,3)$ and end up with the triangular
prism $M_{10}$.
\end{proof}

\begin{lemma}
Let $P \in \M$ be a triangular prism whose triangular facets
are the triangle shown in Figure~\ref{tria1-3}.
Then, $P$ is equivalent to $M_{11}$.
\end{lemma}

\begin{proof}
Without loss of generality we assume that the two triangular
facets of $P$, denoted $F$ and $F'$, are given by
$F := \conv(\{ \pm e_1, 2 e_2\})$ and $F' := a + F$, where
$a = (a_1, a_2, a_3)$ is the integer point in the relative
interior of $F'$.
By applying an appropriate unimodular transformation we can
further assume that $0 \leq a_i < a_3$ for $i = 1,2$.
Since the quadrangular facets of $P$ need to contain integer
points in their relative interior it holds $a_3 \geq 2$.
We now set up the facet description of $P$ which is only
dependent on the parameters $a_1$, $a_2$, and $a_3$.
It follows that $\Z^3 \cap \intt(P)$ is the set of all
$x = (x_1, x_2, x_3) \in \Z^3$ satisfying
\begin{align*}
  2 a_3 x_1 + a_3 x_2 - (2a_1 + a_2 - 1) x_3 & < 2 a_3, &
  - a_3 x_2 & + (a_2 - 1) x_3 < 0, \\
  -2 a_3 x_1 + a_3 x_2 + (2a_1 - a_2 + 1) x_3 & < 2 a_3, &
  x_3 & \in \{1,\ldots, a_3 - 1\}.
\end{align*}
From these inequalities we obtain the following
equivalences:
\begin{itemize} 
  \item $(0,1,1) \in \intt(P)$ if and only if $2a_1 + 1 < a_2 + a_3$;  
  \item $(1,1,1) \in \intt(P)$ if and only if $a_3 + 1 < 2 a_1 + a_2$. 
\end{itemize}
This implies that the following inequalities hold:
\begin{eqnarray}
  a_2 + a_3   & \le & 2 a_1 + 1, \label{sail.1} \\
  2 a_1 + a_2 & \le & a_3 + 1. \label{sail.2}
\end{eqnarray}
Adding \eqref{sail.1} and \eqref{sail.2} yields $a_2 \leq 1$
and therefore $a_2 \in \{0,1\}$. We distinguish into two
cases.

\emph{Case 1}: $a_2 = 0$.
If $2 a_1 > 1$, then $(1,0,1) \in \intt(P)$. 
Thus, we have $2 a_1 \leq 1$ implying $a_1 = 0$.
Substituting this into \eqref{sail.1} leads to $a_3 \leq 1$
which is a contradiction.

\emph{Case 2}: $a_2 = 1$.
From \eqref{sail.1} and \eqref{sail.2}, we obtain $a_3 = 2
a_1$, i.e., $a = (a_1,1,2a_1)$ for some $a_1 \geq 1$.
If $a_1 \geq 2$ we have $(1,1,2) \in \intt(P)$.
Thus, it holds $a = (1,1,2)$ which leads to the triangular
prism $M_{11}$.
\end{proof}

%%%%%%%%%%%%%%%%%%%%%%%%%%%%%%%%%%%%%%%%%%%%%%%%%%%%%%%%%%%%%%%%%%%%%%%
%% Simplices %%
%%%%%%%%%%%%%%%%%%%%%%%%%%%%%%%%%%%%%%%%%%%%%%%%%%%%%%%%%%%%%%%%%%%%%%%

\section{Elements in $\boldsymbol{\M}$ with four facets} \label{four.facets}

Let $P \in \M$ be a simplex and let $F$ be an arbitrary
facet of $P$.
Using a unimodular transformation we can assume that $F
\subseteq \R^2 \times \{0\}$.
Throughout this section we refer to $F$ as the base of $P$
and denote the vertex $a = (a_1, a_2, a_3)$ of $P$ which is
not contained in $\aff(F)$ as the apex of $P$, where we
assume $a_3 > 0$.
We can further assume that $a_3 \geq 2$ since for $a_3 = 1$,
$P$ is contained in the split $\{x \in \R^3 : 0 \leq x_3
\leq 1\}$ which is a contradiction to its maximality.

We first consider simplices $P \in \M$ with $a_3 = 2$ and
$a_3 = 3$, respectively.
Let $F':= P \cap (\R^2 \times \{1\})$.
Since each facet of $P$ contains an integer point in its
relative interior, it follows that $F'$ is a maximal
lattice-free triangle.
Indeed, if $a_3 = 2$, then any integer point $w =
(w_1,w_2,w_3)$ in the relative interior of one of the three
facets different from $F$ satisfies $w_3 = 1$.
On the other hand, if $a_3 = 3$, then any integer point
$w = (w_1,w_2,w_3)$ in the relative interior of one of the
three facets different from $F$ with $w_3 = 2$ guarantees
that the point $2w - a \in F'$ is also an integer point in
the relative interior of the same facet as $w$.
According to Dey and Wolsey \cite{DeyWolsey08} the maximal
lattice-free triangles can be partitioned into three types:
\begin{itemize}
  \item a type 1 triangle, i.e., a triangle with integer
    vertices and exactly one integer point in the relative
    interior of each edge,
  \item a type 2 triangle, i.e., a triangle with at least
    one fractional vertex $v$, exactly one integer point in
    the relative interior of the two edges incident to $v$
    and at least two integer points on the third edge,
  \item a type 3 triangle, i.e., a triangle with exactly
    three integer points on the boundary, one in the
    relative interior of each edge.
\end{itemize}

\begin{figure}[ht]
  \unitlength=1mm
  \begin{center}
    \subfigure[\label{mlf.tr1} Type 1 triangle]{
      \begin{picture}(23,20)
        \put(2,0){\includegraphics[width=20\unitlength]{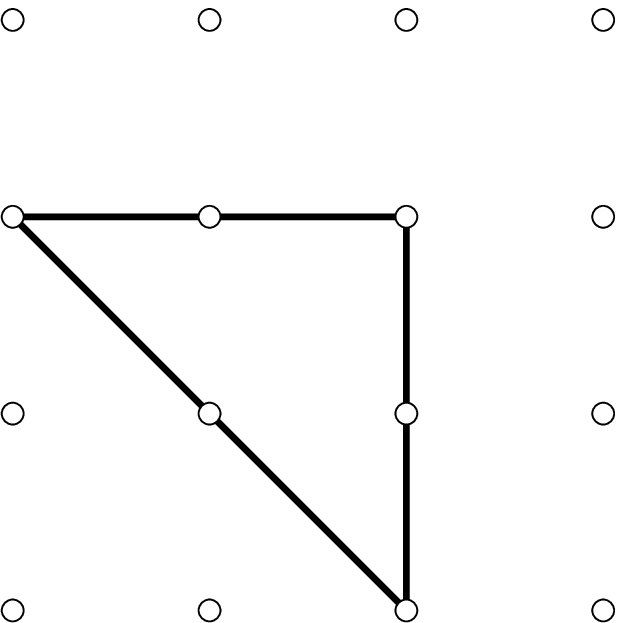}}
      \end{picture}
    } 
    \quad 
    \subfigure[\label{mlf.tr2} Type 2 triangle]{
      \begin{picture}(23,20)
        \put(2,0){\includegraphics[width=20\unitlength]{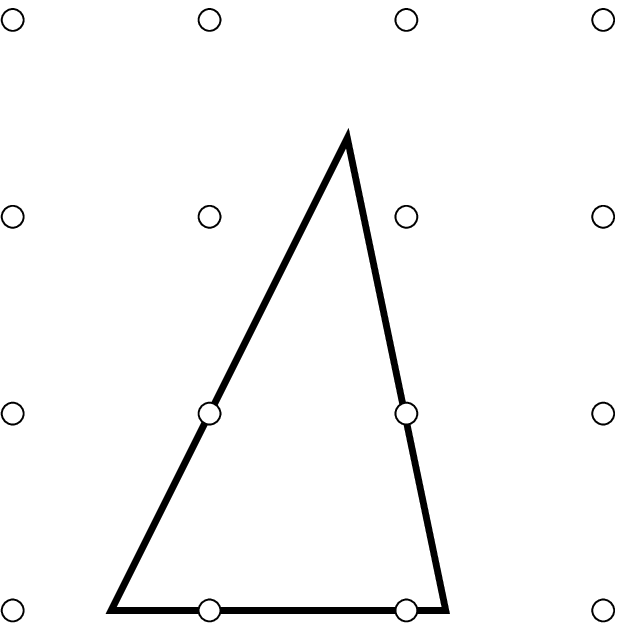}}
      \end{picture}
    }
    \quad
    \subfigure[\label{mlf.tr3} Type 3 triangle]{
      \begin{picture}(23,20)
        \put(2,0){\includegraphics[width=20\unitlength]{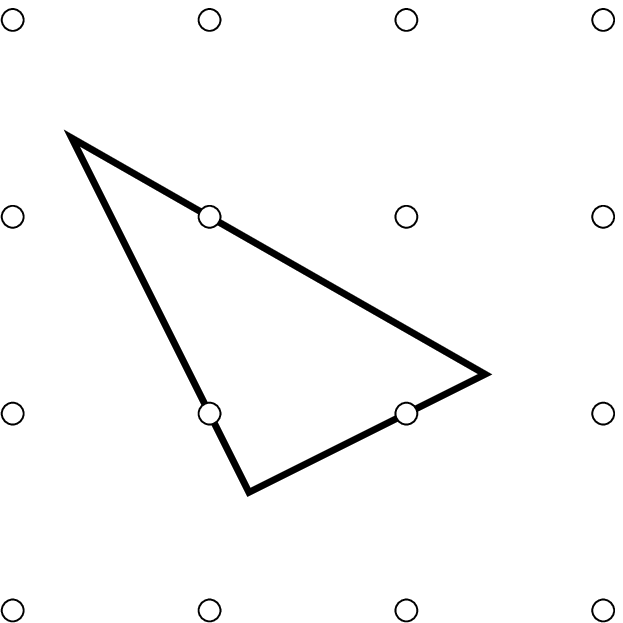}}
      \end{picture}
    } 
  \end{center}
  \caption{\label{mlf.sets} All types of maximal
    lattice-free triangles in dimension two}
\end{figure}

\begin{lemma}
Let $P \in \M$ be a simplex with base $F \subseteq \R^2
\times \{0\}$ and apex $a = (a_1, a_2, a_3)$, where $a_3 \in
\{2,3\}$.
Then $P$ is equivalent to one of the simplices $M_1$, $M_2$,
$M_3$, $M_6$, or $M_7$.
\end{lemma}

\begin{proof}
We distinguish into three cases according to the type of
triangle of $F' := P \cap (\R^2 \times \{1\})$.

\emph{Case 1}: $F'$ is a triangle of type 1.
Without loss of generality assume $F' = \conv(\{e_3, 2e_1 +
e_3, 2e_2 + e_3\})$.
Thus, if $a_3 = 2$, $F$ is a translate of $\conv(\{o, 4e_1,
4e_2\})$ which leads to $M_2$.
If $a_3 = 3$, $F$ is a translate of $\conv(\{o, 3e_1,
3e_2\})$ which leads to $M_3$.

\emph{Case 2}: $F'$ is a triangle of type 2.
Without loss of generality assume that the edge of $F'$
having at least two relative interior integer points
contains the points $(0,0,1)$ and $(0,1,1)$ in its relative
interior, and let the vertex $w = (w_1,w_2,1)$ of $F'$
opposite to this edge satisfy $w_1 > 1$.
By an appropriate unimodular transformation we can assume
that the remaining two edges pass through the points
$(1,0,1)$ and $(1,1,1)$.
First assume $a_3 = 2$.
Then $\vertt(F') = \frac{1}{2}a + \frac{1}{2}\vertt(F)
\subseteq \frac{1}{2} \Z^3$.
Hence, the three vertices of $F'$ lie in $\frac{1}{2} \Z
\times \{\frac{1}{2}\} \times\{1\}$ and $\{0\} \times
\frac{1}{2} \Z \times\{1\}$.
It follows that $F' = \conv(\{(0,\frac{3}{2},1),
(0,-\frac{1}{2},1), (2,\frac{1}{2},1)\})$ or $F' =
\conv(\{(0,2,1), (0,-1,1), (\frac{3}{2},\frac{1}{2},1)\})$.
Thus, in the former case, $F$ is a translate of
$\conv(\{(0,3,0), (0,-1,0), (4,1,0)\})$ leading to $M_7$,
whereas in the latter case $F$ is a translate of
$\conv(\{(0,4,0), (0,-2,0), (3,1,0)\})$ leading to $M_1$.
Now assume $a_3 = 3$.
Then $\vertt(F') = \frac{1}{3}a + \frac{2}{3}\vertt(F)
\subseteq \frac{1}{3} \Z^3$.
Hence, two vertices of $F'$ lie in $\{0\} \times \frac{1}{3}
\Z \times\{1\}$ and the third vertex lies either in
$\frac{1}{3} \Z \times \{\frac{1}{3}\} \times\{1\}$ or
$\frac{1}{3} \Z \times \{\frac{2}{3}\} \times\{1\}$.
By symmetry, we can assume that the third vertex lies in
$\frac{1}{3} \Z \times \{\frac{2}{3}\} \times\{1\}$.
It follows that $F' = \conv(\{(0,2,1), (0,-2,1),
(\frac{4}{3},\frac{2}{3},1)\})$ or $F' =
\conv(\{(0,\frac{4}{3},1), (0,-\frac{2}{3},1),
(2,\frac{2}{3},1)\})$.
Thus, in the former case, $F$ is a translate of
$\conv(\{(0,3,0), (0,-3,0), (2,1,0)\})$ leading to $M_1$,
whereas in the latter case $F$ is a translate of
$\conv(\{(0,2,0), (0,-1,0), (3,1,0)\})$ leading to $M_6$.

\emph{Case 3}: $F'$ is a triangle of type 3.
Without loss of generality assume that $F' =
\conv(\{u,v,w\})$ with $u_1 < 0$, $1 < u_2$, $1 < v_1$, $0 <
v_2 < 1$, $0 < w_1 < 1$, $w_2 < 0$, and $u_3 = v_3 = w_3 =
1$, see Figure~\ref{fig.mlpf.tria}.
\begin{figure}[ht]
  \centering
  \includegraphics[height=4.8cm]{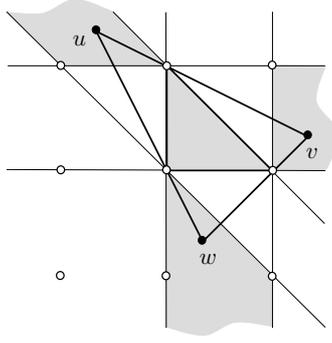}
  \caption{Triangle of type $3$}
  \label{fig.mlpf.tria}
\end{figure}
First assume $a_3 = 2$.
Then $\vertt(F') = \frac{1}{2}a + \frac{1}{2}\vertt(F)
\subseteq \frac{1}{2} \Z^3$.
Thus, it follows $v_2 = w_1 = \frac{1}{2}$ and hence we
obtain $v = (\frac{3}{2},\frac{1}{2},1)$ and $w =
(\frac{1}{2},-\frac{1}{2},1)$.
This implies $u = (-\frac{3}{2},\frac{3}{2},1)$.
However, the edge connecting $u$ and $w$ contains the two
integer points $(0,0,1)$ and $(-1,1,1)$ in its relative
interior which is a contradiction to the fact that $F'$ is
of type $3$.
Now assume $a_3 = 3$.
Then $\vertt(F') = \frac{1}{3}a + \frac{2}{3}\vertt(F)
\subseteq \frac{1}{3} \Z^3$.
Thus, it follows $v_2 \in \{\frac{1}{3},\frac{2}{3}\}$ and
$w_1 \in \{\frac{1}{3},\frac{2}{3}\}$.
Since the edge connecting $v$ and $w$ goes through the point
$(1,0,1)$, the following cases are possible:
\medskip

$ \begin{array}{lllll}
  v = (\frac{5}{3},\frac{1}{3},1),\ w = (\frac{1}{3},-\frac{1}{3},1)
  & \Longrightarrow & u = (-\frac{5}{3},\frac{5}{3},1)
  & \Longrightarrow & F' \mbox{ is of type 2}, \\[2mm]
  v = (\frac{7}{3},\frac{2}{3},1),\ w = (\frac{1}{3},-\frac{1}{3},1)
  & \Longrightarrow & u = (-\frac{7}{6},\frac{7}{6},1)
  & \Longrightarrow & u \not \in \frac{1}{3} \Z^3, \\[2mm]
  v = (\frac{4}{3},\frac{1}{3},1),\ w = (\frac{1}{3},-\frac{2}{3},1)
  & \Longrightarrow & u = (-\frac{2}{3},\frac{4}{3},1),
  & & \\[2mm]
  v = (\frac{5}{3},\frac{2}{3},1),\ w = (\frac{1}{3},-\frac{2}{3},1)
  & \Longrightarrow & u = (-\frac{5}{9},\frac{10}{9},1)
  & \Longrightarrow & u \not \in \frac{1}{3} \Z^3, \\[2mm]
  v = (\frac{4}{3},\frac{2}{3},1),\ w = (\frac{1}{3},-\frac{4}{3},1)
  & \Longrightarrow & u = (-\frac{4}{15},\frac{16}{15},1)
  & \Longrightarrow & u \not \in \frac{1}{3} \Z^3, \\[2mm]
  v = (\frac{4}{3},\frac{1}{3},1),\ w = (\frac{2}{3},-\frac{1}{3},1)
  & & & \Longrightarrow & F' \mbox{ is no triangle}, \\[2mm]
  v = (\frac{5}{3},\frac{2}{3},1),\ w = (\frac{2}{3},-\frac{1}{3},1)
  & \Longrightarrow & u = (-\frac{10}{3},\frac{5}{3},1)
  & \Longrightarrow & (-1,1,1) \in \relintt(F'), \\[2mm]
  v = (\frac{4}{3},\frac{2}{3},1),\ w = (\frac{2}{3},-\frac{2}{3},1)
  & \Longrightarrow & u = (-\frac{4}{3},\frac{4}{3},1)
  & \Longrightarrow & F' \mbox{ is of type 2}.
\end{array} $
\medskip

In seven of these eight cases, it follows that $F'$ is not a
valid triangle.
In the open case where $v = (\frac{4}{3},\frac{1}{3},1)$, $w
= (\frac{1}{3},-\frac{2}{3},1)$, and $ u =
(-\frac{2}{3},\frac{4}{3},1)$ we infer that $F$ is a
translate of $\conv(\{(2,\frac{1}{2},0), (\frac{1}{2},-1,0),
(-1,2,0)\})$.
However, such a translate does never have all three vertices
integer.
\end{proof}

In the following we assume that $a_3 \geq 4$.
Our proof consists of the following steps.
Firstly, we construct all bases which are possible for such
a simplex $P \in \M$.
Secondly, we argue that all simplices $P \in \M$ satisfy
$a_3 \leq 12$.
This gives a finite set of simplices that need to be checked
for maximal lattice-freeness.
Finally, the ultimate list of maximal lattice-free simplices
is obtained by computer enumeration.

By Lemma \ref{2ineq}, all integral triangles $T$ in the
plane with $w(T) = 2$ and $(i(T),b(T)) \in Z_S$ are
potential bases for a maximal lattice-free simplex $P \in
\M$ with $a_3 \ge 4$.
From \eqref{bases.bound}, it follows that $2 i(F) + b(F)
\leq 6$ for $a_3 \geq 11$ and therefore $(i(F),b(F)) =
(1,3)$ or $(i(F),b(F)) = (1,4)$.
If $(i(F),b(F)) = (1,3)$, then $F$ is, up to a unimodular
transformation, the triangle shown in Figure~\ref{tria1-5}.
In Lemma \ref{simpl.kite} we shall show that $a_3 \leq 12$
in this case since otherwise $P$ is not lattice-free.
If $(i(F),b(F)) = (1,4)$, then $F$ is, up to a unimodular
transformation, the triangle shown in Figure~\ref{tria1-3}.
In Lemma \ref{simpl.sail} we shall show that $a_3 \leq 8$ in
this case since otherwise $P$ is not lattice-free.
Thus, we can use computer enumeration to find all simplices
$P \in \M$.

\begin{lemma} \label{triangles.i=2,i=3}
Let $T \subseteq \R^2$ be an integral triangle with $w(T) =
2$, $i(T) = 2$, and $b(T) \in \{3,4,5,6\}$.
Then, up to a unimodular transformation, $T$ is one of the
triangles depicted in Figure~\ref{triangles.i=2}.
\begin{figure}[ht]
  \centering
  \subfigure[\label{tria2-1}]{\includegraphics[height=1.5cm]{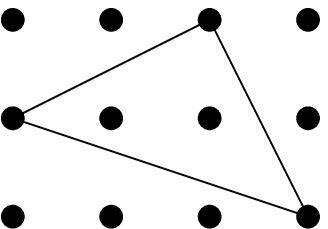}}
  \qquad
  \subfigure[\label{tria2-2}]{\includegraphics[height=1.5cm]{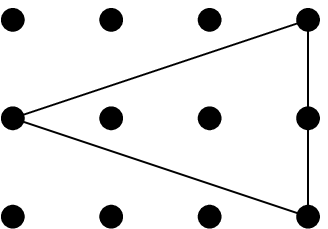}}
  \qquad
  \subfigure[\label{tria2-3}]{\includegraphics[height=1.5cm]{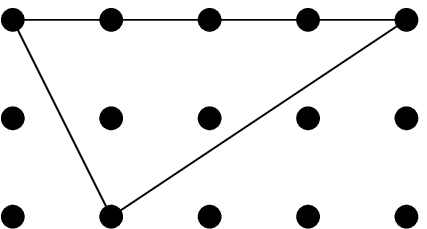}}
  \caption{All integral triangles $T$ with $w(T) = 2$, $i(T)
    = 2$, and $b(T) \in \{3,4,5,6\}$}
  \label{triangles.i=2}
\end{figure}

\end{lemma}

\begin{proof}
Let $T$ be an integral triangle in the plane satisfying
$w(T) = 2$ and $i(T) = 2$.
We divide the proof according to the number of integer
points on the boundary of $T$.

\emph{Case 1:} $i(T) = 2$ and $b(T) = 3$. 
Without loss of generality we assume that the two interior
integer points are placed at $(1,0)$ and $(2,0)$.
This implies that for any $u \in \Z^2 \setminus \{o, \pm
e_2\}$ we have $w(T,u) \geq 3$ and therefore it must hold
$v_2 \in \{0, \pm 1\}$ for each vertex $v = (v_1, v_2)$ of
$T$.
Observe that exactly one vertex of $T = \conv(\{a,b,c\})$
lies on the line $y = 0$, say $a = (0,0)$.
Let the remaining two vertices $b$ and $c$ satisfy $b_2 = 1
= - c_2$.
Using an appropriate unimodular transformation we can assume
that $b = (0,1)$.
For convexity reasons it follows that $c = (5,-1)$ which
leads to the triangle shown in Figure~\ref{tria2-1}.

\emph{Case 2:} $i(T) = 2$ and $b(T) = 4$.
Pick's formula gives $\area(T) = 3$ in this case.
Placing the two interior integer points of $T$ at $(1,0)$
and $(2,0)$ as above implies again that $v_2 \in \{0, \pm
1\}$ for each vertex $v = (v_1, v_2)$ of $T$.
Let $T = \conv(\{a,b,c\})$.
Clearly, we cannot have two vertices on the line $y = 0$.
If none of the vertices is on the line $y = 0$, then assume
without loss of generality that $a_2 = b_2 = 1 = - c_2$.
It follows that either $b_1 = a_1 + 2$ with $\area(T) = 2$,
or $b_1 = a_1 + 1$ with $\area(T) = 1$.
In both cases this is a contradiction to Pick's formula.
Thus, exactly one vertex lies on the line $y = 0$, say $a =
(0,0)$.
Let the remaining two vertices $b$ and $c$ satisfy $b_2 = 1
= - c_2$.
As above, we can assume that $b = (0,1)$ which implies $c =
(6,-1)$.
This gives the triangle shown in Figure~\ref{tria2-2}.

\emph{Case 3:} $i(T) = 2$ and $b(T) = 5$.
Pick's formula gives $\area(T) = 3{.}5$.
Placing the two interior integer points of $T$ at $(1,0)$
and $(2,0)$ as above implies again that $v_2 \in \{0, \pm
1\}$ for each vertex $v = (v_1, v_2)$ of $T$.
Clearly, we cannot have two vertices on the line $y = 0$.
If no vertex of $T$ lies on the line $y = 0$, then with
similar arguments as above we infer that $\area(T) \leq 3$,
a contradiction to Pick's formula.
Thus, precisely one vertex of $T = \conv(\{a,b,c\})$ lies on
the line $y = 0$, say $a = (0,0)$.
Without loss of generality let $b_2 = 1 = - c_2$.
Note that the two edges connecting $a$ and $b$,
resp.~connecting $a$ and $c$, do not have integer points in
their relative interior.
The edge connecting $b$ and $c$ has at most one relative
interior integer point.
Therefore, we have at most four integer points on the
boundary of $T$ which is a contradiction to $b(T) = 5$.

\emph{Case 4:} $i(T) = 2$ and $b(T) = 6$.
Pick's formula gives $\area(T) = 4$.
Placing the two interior integer points of $T$ at $(1,0)$
and $(2,0)$ as above implies again that $v_2 \in \{0, \pm
1\}$ for each vertex $v = (v_1, v_2)$ of $T$.
Clearly, we cannot have two vertices on the line $y = 0$.
If exactly one vertex of $T$ lies on the line $y = 0$, say
$a = (0,0)$, then using the same arguments as above we infer
that $T$ has at most four integer points on its boundary, a
contradiction to $b(T) = 6$.
Thus, no vertex of $T$ is on the line $y = 0$.
Without loss of generality let $a_2 = b_2 = 1 = - c_2$.
It follows that $b_1 = a_1 + 4$, otherwise Pick's formula is
violated.
Using an appropriate unimodular transformation, we obtain
$a = (0,1)$, $b = (4,1)$ and $c = (1,-1)$, see
Figure~\ref{tria2-3}.
\end{proof}

From Lemma \ref{2ineq} and Lemma \ref{triangles.i=2,i=3},
it follows that any facet of a simplex $P \in \M$ with
$a_3 \geq 4$ has the structure shown in Figures
\ref{tria1-1}--\ref{tria1-5} and \ref{triangles.i=2}.
Furthermore, inequalities \eqref{bases.bound} imply that
only \ref{tria1-3} and \ref{tria1-5} are possible if $a_3
\geq 11$.
In the following two lemmas we will show that simplices
having those two bases are not lattice-free for $a_3 \geq
13$.
Thus, by computer enumeration over all potential bases and
values for $a_3$ ranging from $4$ to $12$, we obtain a
finite list of simplices.
Screening those which are not maximal lattice-free we end up
with the simplices $M_4$ and $M_5$.

\begin{lemma} \label{simpl.kite}
Let $P \subseteq \R^3$ be a simplex with one facet being
$\conv(\{e_1, e_2, -(e_1 + e_2)\})$ and apex $a = (a_1, a_2,
a_3) \in \Z^3$, where $a_3 \geq 13$. Then, $P$ is not
lattice-free.
\end{lemma}

\begin{proof}
By applying an appropriate unimodular transformation we can
assume that $0 \leq a_i < a_3$ for $i = 1,2$.
By symmetry, we assume $a_1 \leq a_2$.
We now set up the facet description of $P$ which is only
dependent on the parameters $a_1$, $a_2$, and $a_3$.
It follows that $\Z^3 \cap \intt(P)$ is the set of all
$x = (x_1, x_2, x_3) \in \Z^3$ satisfying
\begin{align*}
  a_3 x_1 - 2 a_3 x_2 + (1 + 2a_2 - a_1) x_3 & < a_3, &
  a_3 x_1 + a_3 x_2 + (1 - &a_1 - a_2) x_3 < a_3, \\
  a_3 x_2 - 2 a_3 x_1 + (1 + 2a_1 - a_2) x_3 & < a_3, &
  x_3 \in \{1,\ldots, &\,a_3 - 1\}.
\end{align*}
From these inequalities, it follows that
\begin{equation} \label{(1,1,1)}
  a_3 + 1 \geq a_1 + a_2,
\end{equation}
since otherwise $(1,1,1) \in \intt(P)$.
Assume $a_1 = 0$.
If $a_2 \leq 1$, we have $(0,0,1) \in \intt(P)$, otherwise
$(0,1,1) \in \intt(P)$.
Therefore, we must have $a_1 \geq 1$.
It follows that
\begin{equation} \label{(0,1,1)}
  2a_1 + 1 \geq a_2,
\end{equation}
since otherwise $(0,1,1) \in \intt(P)$.
Observe that $(0,0,1) \in \intt(P)$ if and only if $a_1 +
a_3 - 2a_2 > 1$ and $a_2 + a_3 - 2a_1 > 1$.
Assume $a_1 \leq 3$.
Then $(0,0,1) \in \intt(P)$:
$a_1 + a_3 - 2a_2 \stackrel{\eqref{(0,1,1)}}{\geq} a_3 -
3a_1 - 2 \geq 2 > 1$;
$a_2 + a_3 - 2a_1 = (a_2 - a_1) + a_3 - a_1 \geq a_3 - a_1
\geq 10 > 1$.
Thus, we have $a_1 \geq 4$.
Using \eqref{(1,1,1)} this implies $a_3 \geq a_2 + 3$ and
therefore $a_2 + a_3 - 2a_1 \geq 2(a_2 - a_1) + 3 > 1$.
Hence, we have
\begin{equation} \label{(0,0,1)}
  2a_2 + 1 \geq  a_1 + a_3,
\end{equation}
since otherwise $(0,0,1) \in \intt(P)$.
However, using inequalities \eqref{(1,1,1)}--\eqref{(0,0,1)}
it can now be shown that $(1,2,3) \in \intt(P)$:
\begin{align*}
  (1,2,3) \in \intt(P) \qquad \Longleftrightarrow \qquad
      3&a_1 + 3a_2 - 2a_3 > 3, \\
      3&a_1 - 6a_2 + 4a_3 > 3, \\
    -~6&a_1 + 3a_2 + \ \, a_3 > 3.
\end{align*}
$3a_1 + 3a_2 - 2a_3 \stackrel{\eqref{(0,0,1)}}{\geq} 5a_1 - a_2 - 2
\stackrel{\eqref{(0,1,1)}}{\geq} 3a_1 - 3 = 3(a_1  -1) > 3$;
$3a_1 - 6a_2 + 4a_3 \stackrel{\eqref{(1,1,1)}}{\geq} 7a_1 - 2a_2 - 4
\stackrel{\eqref{(0,1,1)}}{\geq} 3(a_1 - 2) > 3$;
$-6a_1 + 3a_2 + a_3 \stackrel{\eqref{(1,1,1)}}{\geq} -5a_1 + 4a_2 - 1
\stackrel{\eqref{(0,0,1)}}{\geq} 2a_3 - 3a_1 - 3
\stackrel{\eqref{(1,1,1)}}{\geq} 2a_2 - a_1 - 5
\stackrel{\eqref{(0,0,1)}}{\geq} a_3 - 6 > 3$.
\end{proof}

\begin{lemma} \label{simpl.sail}
Let $P \subseteq \R^3$ be a simplex with one facet being
$\conv(\{\pm e_1, 2e_2\})$ and apex $a = (a_1, a_2, a_3) \in
\Z^3$, where $a_3 \geq 9$. Then, $P$ is not lattice-free.
\end{lemma}

\begin{proof}
By applying an appropriate unimodular transformation we can
assume that $0 \leq a_i < a_3$ for $i = 1,2$.
We now set up the facet description of $P$ which is only
dependent on the parameters $a_1$, $a_2$, and $a_3$.
It follows that $\Z^3 \cap \intt(P)$ is the set of all
$x = (x_1, x_2, x_3) \in \Z^3$ satisfying
\begin{align*}
   2a_3 x_1 + a_3 x_2 + (2 - 2a_1 - a_2) x_3 & < 2a_3, &
            - a_3 x_2 +              a_2 x_3 < 0, \\
  -2a_3 x_1 + a_3 x_2 + (2 + 2a_1 - a_2) x_3 & < 2a_3, &
  x_3 \in \{1, \ldots, a_3 - 1\}.
\end{align*}
From these inequalities we obtain the following
equivalences:
\begin{itemize} 
  \item $(0,1,1) \in \intt(P)$ if and only if $2a_1 + 2 < a_2 + a_3$;  
  \item $(1,1,1) \in \intt(P)$ if and only if $a_3 + 2 < 2a_1 + a_2$. 
\end{itemize}
This implies that the following inequalities hold:
\begin{eqnarray}
  2a_1 + 2 & \geq & \ \, a_2 + a_3, \label{cond.(0,1,1)} \\
   a_3 + 2 & \geq & 2a_1 + a_2. \label{cond.(1,1,1)}
\end{eqnarray}
Adding \eqref{cond.(0,1,1)} and \eqref{cond.(1,1,1)} yields
$a_2 \leq 2$.
Using \eqref{cond.(0,1,1)}, \eqref{cond.(1,1,1)} and $a_2
\leq 2$ it can be shown that $(1,1,2) \in \intt(P)$:
\begin{align*}
  (1,1,2) \in \intt(P) \qquad \Longleftrightarrow \qquad
      4a_1 + 2&a_2 - \ \, a_3 > 4, \\
    -~4a_1 + 2&a_2 + 3a_3 > 4, \\
           -~2&a_2 + \ \, a_3 > 0.
\end{align*}
$4a_1 + 2a_2 - a_3 \stackrel{\eqref{cond.(0,1,1)}}{\geq}
4a_2 + a_3 - 4 > 4$;
$-4a_1 + 2a_2 + 3a_3 \stackrel{\eqref{cond.(1,1,1)}}{\geq}
4a_2 + a_3 - 4 > 4$;
$-2a_2 + a_3 \stackrel{a_2 \leq 2,\, a_3 \geq 9}{>} 4$.
\end{proof}

%%%%%%%%%%%%%%%%%%%%%%%%%%%%%%%%%%%%%%%%%%%%%%%%%%%%%%%%%%%%%%%%%%%%
%% Computer Enumeration  %%
%%%%%%%%%%%%%%%%%%%%%%%%%%%%%%%%%%%%%%%%%%%%%%%%%%%%%%%%%%%%%%%%%%%%

\section{Remarks on the computer enumeration} \label{sect:computer-search}

In view of the results in
Sections~\ref{notions}--\ref{four.facets} for proving Theorem~\ref{main.thm} it remains to
verify the following.
\begin{itemize} 
  \item The integral quadrangular pyramids with bases as in
    Figures~\ref{quad1-3}--\ref{quad1-7} and \ref{i=2} of
    height $h$ with $4 \le h \le 10$ are not in $\M$. 
 \item The integral simplices with bases as in
   Figures~\ref{tria1-1}--\ref{tria1-5} and
   \ref{triangles.i=2} of height $h$ with $4 \le h \le 12$
   belonging to $\M$ are equivalent to $M_4$ or $M_5$.
\end{itemize}
This can be done by a computer enumeration which involves
less than $15\,000$ polytopes.

%%%%%%%%%%%%%%%%%%%%%%%%%%%%%%%%%%%%%%%%%%%%%%%%%%%%%%%%%%%%%%%%%%%%%
%% Literature %%
%%%%%%%%%%%%%%%%%%%%%%%%%%%%%%%%%%%%%%%%%%%%%%%%%%%%%%%%%%%%%%%%%%%%%

\small
%\nocite{*}
%\bibliography{literature}
%\bibliographystyle{amsalpha}

\providecommand{\bysame}{\leavevmode\hbox to3em{\hrulefill}\thinspace}
\providecommand{\MR}{\relax\ifhmode\unskip\space\fi MR }
\providecommand{\MRhref}[2]{%
  \href{http://www.ams.org/mathscinet-getitem?mr=#1}{#2}
}
\providecommand{\href}[2]{#2}

%%%%%%%%%%%%%%%%%%%%%%%%%%%%%%%%%%%%%%%%%%%%%%%%%%%%%%%%%%%%%%%%%%%%%
%% Addresses %%
%%%%%%%%%%%%%%%%%%%%%%%%%%%%%%%%%%%%%%%%%%%%%%%%%%%%%%%%%%%%%%%%%%%%%

\begin{tabular}{l}
  \\
  \\
  \\
  \\
  \textsc{Gennadiy Averkov} \\
  \textsc{Institute of Mathematical Optimization} \\
  \textsc{Faculty of Mathematics} \\
  \textsc{University of Magdeburg} \\
  \textsc{Universit\"atsplatz 2, 39106 Magdeburg} \\
  \textsc{Germany}        \\
  \emph{e-mail}: \texttt{averkov@math.uni-magdeburg.de} \\
  \emph{web}: \texttt{http://fma2.math.uni-magdeburg.de/$\sim$averkov}
  \\ \\
\end{tabular}

\begin{tabular}{l}
  \textsc{Christian Wagner} \\
  \textsc{Institute for Operations Research} \\
  \textsc{ETH Z\"urich} \\
  \textsc{R\"amistrasse 101, 8092 Z\"urich} \\
  \textsc{Switzerland}        \\
  \emph{e-mail}: \texttt{christian.wagner@ifor.math.ethz.ch} \\
  \emph{web}:
  \texttt{http://www.ifor.math.ethz.ch/staff/chwagner}
  \\ \\
\end{tabular}

\begin{tabular}{l}
  \textsc{Robert Weismantel} \\
  \textsc{Institute for Operations Research} \\
  \textsc{ETH Z\"urich} \\
  \textsc{R\"amistrasse 101, 8092 Z\"urich} \\
  \textsc{Switzerland}        \\
  \emph{e-mail}: \texttt{robert.weismantel@ifor.math.ethz.ch} \\
  \emph{web}: \texttt{http://www.ifor.math.ethz.ch/staff/weismant}
\end{tabular}

\end{document}